\documentclass[12pt,a4paper]{article}

\usepackage[naustrian, english]{babel} 
\usepackage[T1]{fontenc}
\usepackage[utf8]{inputenc}
\usepackage{lmodern}
\usepackage{textcomp}
\usepackage{amssymb} 
\usepackage{amsthm}
\usepackage{mathtools}
\usepackage[pdftex]{graphicx}
\usepackage{color}
\usepackage{hyperref}
\usepackage{geometry}
\usepackage{xcolor}

\usepackage{subcaption}

\newcommand{\N}{\mathbb{N}}
\newcommand{\pa}{\partial}
\newcommand{\ve}{\varepsilon}

\newcommand{\vp}{\varphi}
\newcommand{\md}{\mathrm{d}}
\newcommand{\vpa}{\varphi_\ast}

\newcommand{\fa}{f_\ast}
\newcommand{\tf}{\tilde{f}}
\newcommand{\fd}{f^{\downarrow}}
\newcommand{\fad}{f_*^{\downarrow}}
\newcommand{\ha}{h_*}
\newcommand{\D}{\mathcal{D}}
\newcommand{\T}{\mathbb{T}^1}
\newcommand{\h}{\mathcal{H}}
\newcommand{\supp}{\operatorname{supp}}
\newcommand{\xa}{x_{\ast}}
\newcommand{\ms}{\mathcal{S}}
\newcommand{\mx}{x^{\downarrow}}
\newcommand{\mxa}{x_*^\downarrow}

\newcommand{\mk}{\mathcal{K}}
\newcommand{\mt}{\mathcal{T}}
\newcommand{\mM}{\mathcal{M}}

\newtheorem{theorem}{Theorem}
\newtheorem*{theorem*}{Theorem}
\newtheorem{lemma}[theorem]{Lemma}
\newtheorem{proposition}[theorem]{Proposition}
\newtheorem{cor}[theorem]{Corollary}
\newtheorem{rem}{Remark}
\newtheorem{Def}{Definition}     

\begin{document}
	\title{Reversal Collision Dynamics}
	\author{A. Frouvelle\thanks{CEREMADE, UMR 7534, Universite Paris–Dauphine, Place du Marechal de Lattre de Tassigny, 75775 Paris Cedex 16, France. {\tt frouvelle@ceremade.dauphine.fr}} \and L. Kanzler\thanks{CEREMADE, UMR 7534, Universite Paris–Dauphine, Place du Marechal de Lattre de Tassigny, 75775 Paris Cedex 16, France. {\tt laura.kanzler@dauphine.psl.eu}} \and C. Schmeiser\thanks{University of Vienna, Faculty for Mathematics, Oskar-Morgenstern-Platz 1, 1090 Wien, Austria. 
			{\tt christian.schmeiser@univie.ac.at}}}
	\date{}

	\maketitle

	\begin{abstract}
	Motivated by the study of reversal behaviour of myxobacteria, in this article we are interested in a kinetic model for reversal dynamics, in which particles with directions close to be opposite undergo binary collision resulting in reversing their orientations. To this aim, a generic model for binary collisions between particles with states in a general metric space exhibiting specific symmetry properties is proposed and investigated. The reversal process is given by an involution on the space, and the rate of collision is only supposed to be bounded and lower semi-continuous. 
	We prove existence and uniqueness of measure solutions as well as their convergence to equilibrium, using the graph-theoretical notion of connectivity. We first characterise the shape of equilibria in terms of connected components of a graph on the state space, which can be associated to the initial data of the problem. Strengthening the notion of connectivity on subsets for which the rate of convergence is bounded below, we then show exponential convergence towards the unique steady-state associated to the initial condition. The article is concluded with numerical simulations set on the one-dimensional torus giving evidence to the analytical results. 
	\end{abstract}
	
	\medskip 
	

	\medskip
	
	
	\medskip
	
	{\it Acknowledgements:} L.K. and C.S. acknowledge support from the Austrian Science Fund, grants no. W1245 and F65, as well as the ÖAD, mobility grant FR 01/201. L.K. received funding by a grant from the FORMAL team at ISCD - Sorbonne Université.
    A.F. acknowledges support from the Project EFI ANR-17-CE40-0030 of the French National Research Agency. A.F. thanks the hospitality of the Laboratoire de Mathématiques et Applications (LMA, CNRS) in the Université de Poitiers, where part of this research was conducted.
	
	\section{Introduction}\label{intro}
	The motivation of this paper comes from the kinetic equation investigated in \cite{HKMS}, describing a two-dimensional model of myxobacteria updating their direction according to mechanisms of alignment and \emph{reversal collision} with their collision partner: two bacteria with roughly opposite direction of movement may chose to reverse their orientation. We want to study the spatially homogeneous version of this model restricted to reversal collisions only, that we describe now.
	
	We denote by~$\T$ the one-dimensional torus of length~$2\pi$, and for~$\varphi\in\T$ (representing an angle modulo $2\pi$) we write~$\varphi^\downarrow=\varphi+\pi$, the opposite angle. We denote by $b(\varphi,\varphi_*)$ the rate at which particles with angles $\varphi$ and~$\varphi_*$ undergo a \emph{reversal collision} to become, respectively, $\varphi^\downarrow$ and~$\varphi_*^\downarrow$. We assume that the particles are indistinguishable, and therefore~$b(\varphi,\varphi_*)=b(\varphi_*,\varphi)$. As in~\cite{HKMS}, we denote by $d(\cdot,\cdot)$ the distance between two angles on the torus and we assume that two particles only collide when the angle between their orientations is greater than~$\frac{\pi}{2}$, i.e.~$b(\vp,\vpa)= 0$ for $d(\vp,\vpa)\leqslant\pi/2$. We also assume that the collision kernel~$b$ is symmetric with respect to the reversal process: for all angles~$\varphi$ and~$\varphi_*$, we have~$b(\varphi,\varphi*)=b(\varphi^\downarrow,\varphi^\downarrow_*)$.
	The dynamics is then given by the following kinetic equation, describing the evolution in time of a distribution function~$f=f(t,\varphi)\geqslant 0$:
	\begin{align}\label{c4:REVT}
		\pa_t f =  \int_{\T} b(\vp,\vpa)\left(\fd \fad -f\fa\right)\, \md \vpa,
	\end{align}
	where as usual in kinetic theory, we use the notation $\fd=f(t,\varphi^\downarrow)$ and $\fa=f(t,\varphi_*)$. This equation preserves mass, and therefore we expect~$f$ to be a probability density for all time if its initial condition is of mass~$1$. 
	
	In~\cite{HKMS} it was shown that under the assumption that the support of the initial condition~$f_I$ can be separated into two not connected areas, whose distance from each other is at least $\pi/2$, the mass restricted to each of these areas is conserved within time. As a consequence, convergence to an equilibrium with asymmetrically distributed mass in the upper and lower half of the torus $\T$ was observed. On the other hand, if the initial condition does not exhibit such a separation of the support, the equilibrium distribution is proved to be always symmetric.  Additionally, in~\cite{KS} the aforesaid model was investigated with an additional term modelling diffusion with respect to $\vp$. In that case, independent on the initial data, no equilibria with asymmetrically distributed masses in the upper and lower part of the torus could be created. This observation indicated that the special distribution of the support, which is in the latter case destroyed by the diffusive term as time proceeds, is responsible for the final form of the equilibrium the system converges to. These results lead to the fundamental question for this novel type of equation modelling reversal collision dynamics: \emph{How exactly does the support of the initial data encode the shape of the equilibrium?}
	
	Motivated by the aforementioned dynamics on the torus, this question can be posed in a far more general setting. The goal is first to be able to allow for different geometries, for instance if we consider the directions of the particles to be on the unit sphere of~$\mathbb{R}^d$, undergoing reversal collisions whenever they are close to be opposite, say by a defect angle~$\alpha$. Furthermore, we want to look at measure solutions, where the mass may be concentrated at some points. 
	
	We consider a \emph{compact metric space} $\ms$, endowed with its Borel~$\sigma$-algebra (denoted by~$\sigma(\ms)$ when needed), and with a measurable \emph{involution}~$x\in\ms\mapsto\mx$ (that is to say $(\mx)^\downarrow=x$ for all~$x \in \ms$). 
	If~$A$ is a subset of~$\ms$, we denote its \emph{reversed} set by~$A^\downarrow:=\{x^\downarrow,x\in A\}$. We suppose we are given a measurable and nonnegative \emph{collision kernel} $b:\ms\times\ms\to\mathbb{R}_+$, symmetric and invariant by the involution, in the sense that for all~$x,\xa\in\ms$, we have
	\begin{equation}\label{symmetry-b}
		b(x, \xa) = b(\xa, x)=b(\mx,\mxa) \,.
	\end{equation}
	The set of nonnegative bounded measures on $\ms$ will be denoted by $\mM_{\ms}$ and its members will be called \emph{measures on} $\ms$.
	
	Given an initial datum~$f_I\in\mM_{\ms}$ (we do not suppose a priori normalization of the initial mass), we are then interested in the evolution of a time-dependent measure~$f(t)\in\mM_{\ms}$ (which we sometimes, in an abuse of notation, formally identify with its density $f(t,x)$, $x\in\ms$), undergoing reversal collisions
	\[(x,\xa) \longrightarrow (\mx,\mxa),\]
	with interaction rate given by~$b(x,\xa)\geqslant0$. 
	
	\begin{Def}\label{def-rev-coll-measure}
		We say that~$f \in C([0,T); \mathcal{M}_{\ms})$ is, on the time interval $[0,T)$, a solution of the reversal collision dynamics on $\ms$ with initial condition~$f_I$, if for any Borel set~$A\in\sigma(\ms)$ the integral~$\int_A \md f(t)$ is a continuous function of time with initial value~$\int_A \md f_I$, differentiable on~$(0,T)$ and satisfying 
		\begin{equation}
			\label{eq-rev-coll-measure}
			\frac{\md}{\md t}\int_A \md f = \iint_{A^\downarrow\times\ms} b(x,\xa) \md f\, \md \fa -\iint_{A\times\ms} b(x,\xa) \md f\, \md \fa \, \text{ on }(0,T). 
		\end{equation}
	\end{Def}
	
	The main object of this paper is to characterize the long-time behaviour of solutions of the reversal collision dynamics~\eqref{eq-rev-coll-measure}, according to connectivity properties of the support of the initial condition~$f_I$. Our main result, Theorem~\ref{thm-asymptotic-behaviour}, gives that under minimal assumptions the solution converges exponentially fast to a steady-state which can be easily described from the initial condition~$f_I$. A simple corollary of this result in the case where~$b(x,x^\downarrow)>0$ for all~$x\in\ms$ as in our motivating example reads as follows (this is a reformulation of Corollary~\ref{corollary-main-theorem}):
	\begin{theorem*}
		Let~$f_I$ be a probability measure on~$\ms$ and~$\mu=\frac12(f_I+f_I^\downarrow)$. We suppose~$b$ is lower-semicontinuous and bounded, and such that~$b(x,x^\downarrow)>0$ for all~$x\in\ms$. Then there exists a unique global in time solution~$f(t,\cdot)$ regarding Def. \ref{def-rev-coll-measure} to the reversal collision dynamics, and there exists a finite number of sets~$\mt_i$ such that $\bigcup_i \mt_i =  \supp(\mu)$ (only depending on~$b$ and~$\mu$) which are compact and such that for all~$i$,
		\begin{itemize}
			\item either~$\mt_i\neq\mt_i^\downarrow:=\{x^\downarrow:\, x\in \mt_i\}$ and~$f$ converges exponentially fast (in total variation distance) to~$(1+\eta_i)\mu$ on~$\mt_i$ and~$(1-\eta_i)\mu$ on~$\mt_i^\downarrow$, where the constant~$\eta_i\in[-1,1]$ and is given by~$\eta_i=\frac{\int_{\mt_i}\md f_I}{\int_{\mt_i} \md \mu}-1$,
			\item or~$\mt_i=\mt_i^\downarrow$ and in that case~$f$ converges exponentially fast (in total variation distance) to~$\mu$ on~$\mt_i$.
		\end{itemize}
		Furthermore, the rate of convergence only depends on~$b$ and~$\mu$, and~$f$ is zero outside the sets~$\mt_i$ for all time.
	\end{theorem*}
	We do not enter in details here in the construction of the sets~$\mt_i$, but they will be given as connected components of a graph of interaction which can be easily determined from~$\mu$ and~$b$.

	The motivating kinetic equation studied in~\cite{HKMS} also includes alignment, which consists in jumps in the angle variable towards the average direction of a pair of interacting particles (in the same fashion as in the so-called BDG alignment model~\cite{BDG,CDW}, but without directional noise). For special initial configurations, it is shown that the distribution of angles concentrates on antipodal Dirac masses. In future work, we may hope to combine the results of the present paper with the ideas from~\cite{DFR} in which, for alignment only, the Dirac masses are shown to be locally asymptotically stable. 
	
	The model \eqref{c4:REVT} we study in this article can be related to other many-particle models at the mesoscopic scale from mainly biological context, where reversal interactions of many particle systems are considered as well. After this kind of local-reversal operator was introduced in~\cite{HKMS}, its effect in combination with \emph{alignment} of individuals as well as directional diffusion was further studied in \cite{KS}. 
	Moreover, with biological motivation again coming from the \emph{rippling phenomenon} within colonies of myxobacteria \cite{IO}, in \cite{DMY1} the authors introduced and investigated a spatially heterogeneous model of mean-field type, i.e. individual's interactions are modelled as a non-local process. There,
	reversals of cells either depend on the density of the agents moving in reverse direction or occur spontaneous. It turned out to be crucial to include a waiting time between reversals of individuals to be able to see the rippling-wave patterns \cite{DMY2, IO, IMWKO}. A different context of the importance of reversals of cells is given in \cite{EP}, where protrusions and retractions in the movement of polarized cells are studied. The models are based on individual-cell dynamics, where the switching of direction of an agent is modelled by a probability, depending on the microscopic “steps” the agent did in that direction. From there, the authors derive a kinetic-renewal system and further study the relevant macroscopic limits in various scenarios of different complexity.  
	
	The structure of the paper is the following: in Section~\ref{propertiesQRev}, we prove existence and uniqueness of measure solutions to the reversal collision dynamics~\eqref{eq-rev-coll-measure}. In Section~\ref{c4:asy}, we provide the proof of our main result regarding the characterization of the asymptotic steady-state and the exponential rate of convergence. In Section~\ref{section-torus} we apply this result to detail the case of the one-dimensional torus and finally in Section~\ref{c4:numerics} we provide numerical simulations of this specific example.
	
	\section{Properties of solutions, existence and uniqueness}
	\label{propertiesQRev}
	
	In this section, we describe properties of solutions to the reversal collision dynamics operator, and further use them to prove existence and uniqueness in case of a bounded kernel~$b$.
	
	When~$h$ is a function on~$S$, we denote~$h^\downarrow$ the function given by~$h^\downarrow(x)=h(x^\downarrow)$. Similarly, when~$f$ is a measure on~$\ms$, we denote~$f^\downarrow$ the measure such that for any Borel set~$A$ of~$\ms$, we have
	\[\int_A\md f^\downarrow = \int_{A^\downarrow} \md f.\]
	We start by observing several invariance properties of the reversal collision dynamics. 
	\begin{proposition}Let $f_I$ be a measure on~$\ms$ with total mass~$\rho$ and~$\mu:=\frac12(f_I+f_I^\downarrow)$ its symmetric part. Then for any measure solution~$f(t)$, defined as in Def. \ref{def-rev-coll-measure}, to the reversal collision dynamics~\eqref{eq-rev-coll-measure}, 
		\begin{itemize}
			\item[(i)] the measure~$\frac12(f(t)+f(t)^\downarrow)$ is constant in time (and therefore equal to~$\mu$).
			\item[(ii)] for any symmetric Borel set~$A$ (i.e. $A^\downarrow=A$), the quantity~$\int_A\md f(t)$ is constant in time (and therefore equal to~$\int_A\md \mu$).
			\item[(iii)] the total mass is conserved:
			\begin{equation*}\int_{\ms} \md f(t)=\int_{\ms} \md f_I=\rho.
			\end{equation*}
			\item[(iv)] $t \mapsto \frac{1}{\rho}f(\frac{t}{\rho},\cdot)$ is a solution to the reversal collision dynamics and ~or any time $t$ it is a probability measure on $\ms$ (with initial condition $\frac1{\rho}f_I$).
		\end{itemize}
	\end{proposition}
	
	\begin{proof}(i) This is an obvious consequence of $(A^\downarrow)^\downarrow = A$, which can be seen from formulation \eqref{eq-rev-coll-measure} using the symmetry of the collision kernel \eqref{symmetry-b}.\\
	(ii) This follows immediately from formulation \eqref{eq-rev-coll-measure}, since on a symmetric set $f$ is equal to its symmetric part.\\
	(iii) Application of (ii) with $A=\ms$.\\
	(iv) This is a straightforward consequence of the fact that the collision operator is quadratic.
	\end{proof}
	
	By the symmetry~\eqref{symmetry-b} the conserved even part~$\mu$ of $f$ is also a stationary solution to the reversal collision dynamics~\eqref{eq-rev-coll-measure}. We shall therefore restrict our attention to the odd part $f-\mu$, satisfying a linear problem, and also to choosing probability measures as initial data (as a consequence of (iv)).
	
	\begin{theorem} \label{thm-existence-uniqueness} Let~$f_I$ be a probability measure on~$\ms$, let $\mu=\frac12(f_I+f_I^\downarrow)$, and let the collision kernel~$b\geqslant 0$ be measurable and bounded on $\ms\times\ms$.
		
		Then there exists a unique solution~$f$ to the reversal collision dynamics~\eqref{eq-rev-coll-measure} on~$[0,\infty)$ with initial condition~$f_I$, which can be written as~$f=(1+h)\mu$, where~$h\in C([0,\infty),L^\infty( \mu))$ solves
		\begin{equation}
			\partial_t h = -2 \int_{\ms}b(x,x_*)(h+h_*)\md \mu_* \,.\label{lin-eq-h}
		\end{equation}
		Furthermore we have~$-1\leqslant h \leqslant1$ in $\ms\times [0,\infty)$.
	\end{theorem}
	
	\begin{proof} By the nonnegativity of $f_I$ and therefore also of $f_I^\downarrow$, we have $f_I\leqslant2\mu$, implying that the odd part $f_I-\mu$ of $f_I$ is absolutely continuous with respect to $\mu$. We denote its Radon-Nikodym derivative by $h_I$. By the evenness of $\mu$ and the oddness of $f-\mu$, it is an odd function, i.e.~$h_I^\downarrow = -h_I$, and it satisfies $|h_I|\leqslant1$ since, by $\md f_I = (1+h_I)\md\mu$ and by the nonnegativity of~$f_I$ and~$\mu$ it satisfies $h_I\geqslant -1$ (and therefore, by oddness, $h_I\leqslant1$). 
	
	Since the right hand side of \eqref{lin-eq-h} is a bounded linear operator on $L^\infty(\mu)$, the initial value problem for \eqref{lin-eq-h} with the initial condition $h(t=0) = h_I$ has a unique solution in $C([0,\infty),L^\infty(\mu))$. The solution could be constructed by Picard iteration on the mild formulation
	\begin{equation}\label{eq-fixed-point}
	   h = e^{-\gamma t}h_I - 2\int_0^t e^{\gamma(s-t)} 
	   \int_{\ms} b(x,x_*)h_*(s) d\mu_* ds \,,\qquad 
	   \gamma(x) = 2\int_{\ms} b(x,x_*)d\mu(x_*) \,,
	\end{equation}
	which is easily seen to propagate the bounds $-1\leqslant h\leqslant1$. It is easily computed that~$f=(1+h)\mu$ fulfils equation \eqref{eq-rev-coll-measure}. Moreover, due to the continuity of $h$ with respect to time this also holds for $f$, from which one can conclude that all properties of Def. \ref{def-rev-coll-measure} are satisfied. Hence,~$f=(1+h)\mu$ is a solution to the reversal collision dynamics.
	
	Concerning uniqueness: For any measure solution $f$ of \eqref{eq-rev-coll-measure} we can use the argument used for $f_I$
	above to assert the existence of $h$ with the desired properties,
	such that $f=(1+h)\mu$. Substitution in \eqref{eq-rev-coll-measure}
	gives \eqref{lin-eq-h}, such that $h$ has to be the unique solution of \eqref{lin-eq-h} derived above.
\end{proof}
	
	\begin{rem} In Theorem~\ref{thm-existence-uniqueness}, the topological nature of~$\ms$ does not play any role. The result is therefore valid for any measurable space~$\ms$. However, once we have a compact metric space~$\ms$, the space of probability measures has a natural topology of weak convergence and it may be of interest to know if the solution of the reversal collision dynamics is continuous with respect to the initial condition. Under the hypothesis of Theorem~\ref{thm-existence-uniqueness}, and with no more restriction on the collision kernel~$b$, this is not the case. Indeed, in our motivating example, taking~$\ms=\T$ and~$b(\varphi,\varphi^\downarrow)=1$ if~$d(\varphi,\varphi^\downarrow)>\frac{\pi}2$
		and~$b(\varphi,\varphi^\downarrow)=0$ otherwise, we look at the family of initial conditions~$f^\varepsilon_{I}=\frac12(\delta_0+\delta_{\frac{\pi}2+\varepsilon})$. The solution is given, for~$\varepsilon>0$ by
		\[f^\varepsilon=\frac14\big((1+e^{-t})(\delta_0+\delta_{\frac{\pi}2+\varepsilon})+(1-e^{-t})(\delta_\pi+\delta_{-\frac{\pi}2+\varepsilon})\big).\]
		However, the solution for the case~$\varepsilon=0$ is the constant measure $f^0=f_{I}^0$ for all time. The family of initial conditions~$f_{I}^\varepsilon$ converges weakly to $f_{I}^0$ when~$\varepsilon\to0$, but at any given fixed time~$t>0$, the solution~$f^\varepsilon$ does not converge weakly to~$f^0$ when~$\varepsilon\to0$. 
		
		If we suppose now that the collision kernel~$b$ is Lipschitz continuous, we prove in Appendix~\ref{appendix-wasserstein} that we recover the well-posedness in the sense of Hadamard for the topology of weak convergence.
	\end{rem}
	
	\section{Asymptotic behavior}\label{c4:asy}
	In this section, we fix a probability measure~$f_I$ on~$\ms$ and denote by~$f$ the solution to the reversal collision dynamics, given by Theorem~\ref{thm-existence-uniqueness}. We still denote by~$\mu$ the symmetric part of~$f_I$ and by~$h\in C([0,+\infty),L^\infty(\mu))$ the function such that~$f=(1+h)\mu$.
	
	The variance of the anti-symmetric part of $f$ in the probability space determined by $\mu$ is then given by
	\begin{equation}\label{c4:H}
		\h[f]= \frac{1}{2}\int_{\ms} h^2 \, \md \mu = \frac{1}{4}\iint_{\ms\times\ms}  (h-h_*)^2 \, \md \mu \,\md \mu_* \,. 
	\end{equation}
		Since~$h$ and $\partial_th$ are uniformly bounded ($\mu$-a.e.), we can interchange integration and derivation in time, and we obtain using \eqref{lin-eq-h}
	\begin{align*}
		\frac{\md}{\md t} \h[f] &= \int_{\ms}h \,\partial_th \md \mu=-2 \iint_{\ms\times\ms} b(x,x_*) h(h+\ha) \, \md \mu\, \md \mu_*  \\
		&=- \iint_{\ms\times\ms} b(x,x_*) (h+\ha)^2 \, \md \mu \,\md \mu_* =: - \D[f] , 
	\end{align*}
	where the last equality is due to symmetry. Therefore the quantity~$\mathcal{H}[f]$ is nonincreasing in time and bounded below and, thus, convergent as~$t\to+\infty$. By differentiating in time once more, using the fact that~$b$ is bounded and~$h$ is uniformly bounded in time, we obtain that the second derivative of~$\mathcal{H}[f]$ is uniformly bounded in time, and this classically ensures that the derivative of~$\mathcal{H}[f]$  converges to~$0$ as~$t\to+\infty$. This is an indication that the solution may converge to a state such that~$\D[f]=0$ (and we will prove that this is indeed equivalent to be a steady-state).
	
	From the expression for~$\D[f]$ we expect that in equilibrium collision partners carry opposite values of $h$, whereas two elements having a common collision partner must have the same value of~$h$. This motivates the definitions below. 
	We first denote by $\mk$ the support of $\mu$ defined in the sense of measures:
	\begin{equation*}
		\mk:=\supp{(\mu)}:= \big\{x \in \ms: \, \forall \ve >0, \int_{B_{\ve}(x)}\md \mu > 0, \,\, \big\},
	\end{equation*}
	which is compact, since it is a closed subspace of the compact space $\ms$.
	
	\begin{Def} \label{defGraphGamma}
		For~$x,x_*\in\mk$, we say that $x$ and~$x_*$ are \underline{collision partners} whenever~$b(x,x_*)>0$. 
		\begin{itemize}
			\item[(i)] For a set~$\mt\subset\mk$, we denote by~$\mt_*$ the \underline{set of collision partners} of elements of~$\mt$: 
			\begin{equation*} 
				\mt_*:=\{x_*\in\mk, \exists x\in\mt, b(x,x_*)>0\}.
			\end{equation*}
			\item[(ii)]
			 We say $x\in\mk$ and $y\in\mk$ are \underline{adjacent} and write~$x \longleftrightarrow y$ when they have a common collision partner, i.e. there exists~$x_* \in \mk$ such that~$b(x,x_*)>0$ and~$b(y,x_*)>0$.
			This relation defines the graph $\Gamma = (\mk,\mathcal{E})$, where the vertices are the elements of $\mk$ and the edges are given by $\mathcal{E}:=\{(x,y) \in \mk \times \mk: \,\, x\longleftrightarrow y\}$.
			\item[(iii)] The \underline{connected component} of~$x\in\mk$ is the set of all~$y\in\mk$ such there exists a path $(x=x_0,x_1,\dots,x_n=y)$ of adjacent nodes with~$x_{i}\longleftrightarrow x_{i+1}$ for $0\leqslant i<n-1$.
		\end{itemize}
	\end{Def}
	The following result provides a decomposition of the reversal collision dynamics.
	
	\begin{proposition}\label{prop-decomposition}
	Let the assumptions of Theorem \ref{thm-existence-uniqueness} hold and let~$b$ be lower semi-continuous. Let $f(t)$, $t\geqslant 0$, be the solution of \eqref{eq-rev-coll-measure} and let~$\mt$ be a connected component of the graph~$\Gamma$.
	\begin{itemize}
	    \item[(i)] If $\mt_*$ is empty, $\mt$ consists of an isolated node $x$ of $\Gamma$, $\mt=\{x\}$, and 
	    \[\int_{\{x\}} \md f(t) = \int_{\{x\}} \md f_I \,.\]
	    \item[(ii)] If $\mt_*$ is nonempty, then $\mt$, $\mt^\downarrow$, $\mt_*$, and $(\mt_*)^\downarrow$ are connected components of~$\Gamma$ and open sets in~$\mk$, satisfying $(\mt_*)_*=\mt$ and $\mt_*^{\downarrow}:= (\mt_*)^\downarrow = (\mt^\downarrow)_*$.\\
	    Let $\mathcal{A} :=\mt\cup\mt^\downarrow\cup\mt_*\cup\mt^{*\downarrow}$. Then $f|_{\mathcal{A}}$ is a solution to the reversal collision dynamics on~$\mathcal{A}$.
	\end{itemize}
	\end{proposition}
	
	\begin{proof}
	(i) is obvious.\\
	(ii) Let $\mt$ be such that $\mt_*$ is nonempty, and let $x_*,y_*\in\mt_*$. Then there exist $x,y\in\mt$ such that~$b(x,x_*), b(y,y_*)>0$. By connectedness of $\mt$ there exists a path 
	\[
	   x=x_0 \leftrightarrow x_1 \leftrightarrow \cdots \leftrightarrow x_N = y \,,
	\]
	and, consequently, there exists $x_{*,j}\in \mt_*$ such that $b(x_{j-1},x_{*,j}), b(x_j,x_{*,j})>0$, $j=1,\ldots,N$. This implies
	\[
	  x_* \leftrightarrow x_{*,1} \leftrightarrow \cdots \leftrightarrow x_{*,N} \leftrightarrow y_* \,,
	\]
	and therefore \underline{connectedness of $\mt_*$}. The reflection invariance of $b$ implies \underline{connectedness of $\mt^\downarrow$} \underline{and of $(\mt_*)^\downarrow$}.\\
	By the argument above, connectedness of $\mt_*$ implies connectedness of $(\mt_*)_*$. Non-emptiness of~$\mt_*$ implies the existence of $x\in\mt$ with a collision partner in $\mt_*$, and therefore $x\in\mt\cap(\mt_*)_*$, implying \underline{$(\mt_*)_* = \mt$}. \\
	Let $x\in(\mt_*)^\downarrow$. This is equivalent to $x^\downarrow\in\mt_*$, which is again equivalent to the existence of a collision partner $y\in\mt$ of $x^\downarrow$. By the reflection invariance of $b$ this is equivalent to the existence of a collision partner $y^\downarrow\in\mt^\downarrow$ of $x$. Finally this is equivalent to $x\in(\mt^\downarrow)_*$, proving~\underline{$(\mt_*)^\downarrow = (\mt^\downarrow)_* =: \mt_*^\downarrow$}. \\
	The results so far imply that the roles of the four sets $\mt,\mt_*, \mt^\downarrow, \mt_*^\downarrow$ can be interchanged. For example,
	$\hat\mt := \mt_*^\downarrow$ implies $\hat\mt_* = \mt^\downarrow$, $\hat\mt^\downarrow = \mt_*$, and $\hat\mt_*^\downarrow = \mt$. Therefore it is sufficient to prove that one of the for sets is open in $\mk$.\\
	For every $x_*\in\mt_*$ there exists $x\in\mt$ such that $b(x,x_*)>0$. By lower semicontinuity of~$b$ this implies $b(x,y_*)>0$ for all $y_*$ in a neighborhood of $x_*$, implying $y_*\in\mt_*$ and therefore \underline{openness of $\mt_*$}. \underline{Openness of $\mt,\mt^\downarrow$, and~$\mt_*^\downarrow$} follows from the remark above, which also implies~$\mathcal{A}=\mathcal{A}_*=\mathcal{A}^\downarrow=\mathcal{A}_*^\downarrow$. Thus, for $A\subset\mathcal{A}$, the right hand side of \eqref{eq-rev-coll-measure} depends only on $f|_{\mathcal{A}}$, completing the proof.
	\end{proof}
	
	The next step is the identification of equilibria of the dynamics on sets of the form of $\mathcal{A}$. Therefore, for every set $\mt$ we define its mass and its average by
	\begin{align}\label{def:massT}
		\rho:=\int_{\mt}\md \mu\,,\quad \langle h\rangle_{\mt}=\frac1{\rho}\int_{\mt}h\,\md \mu,
	\end{align}
	and similarly for $\mt_*$, as well as the quantity
	\begin{align}\label{def:constT}
		\eta_{\mt}:=\frac{\rho\langle h\rangle_{\mt}-\rho_*\langle h\rangle_{\mt_*}}{\rho+\rho_*}.
	\end{align} 
	
	\begin{proposition}\label{prop-steadystates} Let the assumptions of Proposition \ref{prop-decomposition} hold and let~$\mathcal{A} :=\mt\cup\mt^\downarrow\cup\mt_*\cup\mt_*^{\downarrow}$ with~$\mt_*$ nonempty.
		\begin{itemize}
			\item[(i)] The measure $\mu$ assigns positive mass to all four parts of $\mathcal{A}$, in particular $\rho>0$ and $\rho_*>0$. The average $\langle h\rangle_{\mt}$ satisfies that
			\begin{equation}
				\eta_{\mt}=\frac{\rho\langle h\rangle_{\mt}-\rho_*\langle h\rangle_{\mt_*}}{\rho+\rho_*}=\frac{\int_{\mt\cup\mt_*^{\downarrow}}h\,\md \mu}{\int_{\mt\cup\mt_*^{\downarrow}}\md \mu}=\frac{\int_{\mt\cup\mt_*^{\downarrow}}\md f_I}{\int_{\mt\cup\mt_*^{\downarrow}}\md \mu}-1 \in [-1,1] \label{eq-eta}
			\end{equation}
			is independent of $t$.
			\item[(ii)] The following conditions are equivalent
			\begin{itemize}
				\item $f|_\mathcal{A}$ is a steady state of the reversal collision dynamics on~$\mathcal{A}$,
				\item $\mathcal{D}_{\mt}[f]:=\iint_{\mt\times\mt_*}   b(x,\xa) (h+\ha)^2 \, \md \mu \,\md \mu_*=0$,
				\item $h=\eta_{\mt}$ on~$\mt\cup\mt_*^\downarrow$ and~$h=-\eta_{\mt}$ on~$\mt_*\cup\mt^\downarrow$,  $\mu$-almost everywhere.
			\end{itemize}
		\end{itemize}
	\end{proposition}
	\begin{proof} (i) By the openness of $\mt$ in $\mk$, for $x\in\mt$ there exists a ball $B(x,\delta)$ such that $B(x,\delta)\cap\mk$ is included in~$\mt\cap\mk$. Therefore
	\[
	  \rho \geqslant \int_{B(x,\delta)}\md\mu >0 \,,
	\]
	since $x$ is in the support of $\mu$. In the same way: $\rho_*>0$.\\
	Integration of \eqref{lin-eq-h} over $\mt$ and over $\mt_*$ gives
		\begin{equation*}
			\frac{\md}{\md t}\int_{\mt}h\,\md \mu = \frac{\md}{\md t}\int_{\mt_*}h\,\md \mu 
			=-2\iint_{\mt\times\mt_*} b(x,\xa)  (h+h_*) \md \mu_* \md \mu \,,
		\end{equation*}
		proving that $\eta_{\mt}$ is constant in time. The second and third equalities in \eqref{eq-eta} are due to the oddness of $h$ and the definition of $\mu$. Note that by connectedness $\mt$ and $\mt_*^\downarrow$ are either equal or disjoint and that the formulas hold in both cases. Finally, $\eta_{\mt} \in[-1,1]$ follows from $h\in[-1,1]$.\\
		(ii) If~$f|_{\mathcal{A}}$ is a steady state, then~$\mathcal{D}[f|_\mathcal{A}]=0$ and therefore~$\mathcal{D}_{\mt}[f]=0$. 
		If this is the case, then the quantity~$\int_{\mt_*}b(x,x_*)(h(x)+h(x_*))^2\md \mu_*$ is zero for~$\mu$-almost every~$x\in\mt$. Since~$\mt$ is of positive mass, then there exists such a~$x$ and we denote~$\eta_0=h(x)$, and therefore~$b(x,x_*)(\eta_0+h(x_*))^2=0$ for~$\mu$-almost every~$x_*$ in~$\mt_*$. Now if~$b(x,x_*)>0$, by lower semi-continuity there exists~$\delta>0$ such that for all~$y_*\in\mk\cap B(x_*,\delta)$ and for all~$y\in\mk\cap B(x,\delta)$, we have~$b(y,y_*)>0$ (and then~$y\in\mt$ and~$y_*\in\mt_*$). Therefore~$h(y_*)=-\eta_0$ for~$\mu$-almost every~$y_*\in B(x_*,\delta)\cap\mt_*$. And then we have
		\begin{align*}
			0& = \iint_{\mt\cap B(x,\delta)\times\mt_*\cap B(x_*,\delta)}b(y,y_*)(h(y)+h(y_*))^2\md \mu(y)\md \mu(y_*)\\
			&=\iint_{\mt\cap B(x,\delta)\times\mt_*\cap B(x_*,\delta)}b(y,y_*)(h(y)-\eta_0)^2\md \mu(y)\md \mu(y_*).
		\end{align*}
		Once more since the mass of~$\mt\cap B(x_*,\delta)$ is positive (this is the same as~$\mk\cap B(x_*,\delta)$ and~$\mk$ is the support of~$\mu$), there exists then a~$y_*\in B(x_*,\delta)$ such that~$\int_{\mt\cap B(x,\delta)}b(y,y_*)(h(y)-\eta_0)^2\md \mu(y)$ is zero. And from here we conclude that~$h=\eta_0$~$\mu$-almost everywhere on~$B(x,\delta)\cap\mt$. What we have in general is therefore the following: when~$b(y,y_*)>0$, if~$h=\eta_0$~$\mu$-a.e. on a neighborhood of~$y$, then~$h=-\eta_0$~$\mu$-a.e. on a neighborhood of~$y_*$ (and conversely). 
		
		We want to prove that~$h=\eta_0$ on~$\mt$ ($\mu$-a.e.). Since~$\mk$ is compact, then it is separable, and therefore~$\mt$ is also separable and we only need to prove that~$h=\eta_0$ ($\mu$ a.e.) in the neighborhood of any point of~$\mt$. We fix~$y\in\mt$ and we take a path~$(x=x_0,x_1,\dots,x_n=y)$ of consecutive adjacent elements in~$\mt$. By induction using the previous property, we get that~$h=\eta_0$~$\mu$-a.e on the neighborhood of~$x_i$ for all~$1\leqslant i\leqslant n$, and therefore in the neighborhood of~$y$. Finally we also have that for all~$y_*\in\mt_*$,~$h=-\eta_0$~$\mu$-a.e on the neighborhood of~$y_*$. Therefore we conclude that~$h=-\eta_0$ on~$\mt_*$ ($\mu$-a.e.), and we obtain~$\int_{\mt}h\md \mu-\int_{\mt_*}h\md \mu=(\rho+\rho_*)\eta_0$, and therefore we get~$\eta_0=\eta_{\mt}$. 
		
		Finally, if~$h$ has this form and~$A\subset\mt$ is a Borel set, integration of \eqref{lin-eq-h} over $A$ gives
		\begin{equation*}
			\frac{\md}{\md t}\int_{A}h\,  \md \mu =-2\iint_{A\times\mt_*} b(x,\xa)  (h+h_*) \md \mu_* \md \mu=0 \,.
		\end{equation*}
		We proceed similarly and get the same result when~$A$ is a Borel set included in~$\mt_*$, and by symmetry when~$A$ is included in~$\mt^\downarrow$ or~$\mt_*^\downarrow$. At the end, for any Borel set~$A\subset\mathcal{A}$, the integral~$\int_Ah\md \mu$ is constant in time, therefore~$f|_\mathcal{A}$ is a steady state.
	\end{proof}
	
	\begin{rem}\label{rem-five-cases}
		In Proposition~\ref{prop-decomposition} the connected components~$\mt$,~$\mt_*$,~$\mt^\downarrow$ and~$\mt_*^\downarrow$ do not need to be disjoint, and thus we have the following five possibilities:
		\begin{itemize}
			\item[(i)] The four sets are disjoint.
			\item[(ii)] $\mt=\mt_*^\downarrow\ne\mt_*=\mt^\downarrow$. In this case \eqref{eq-eta} simplifies to~$\eta_{\mt}=\frac{\int_\mt\md f_I}{\int_\mt\md \mu}-1$.
			\item[(iii)] $\mt=\mt_*\neq\mt^\downarrow=\mt_*^\downarrow$. 
			\item[(iv)] $\mt=\mt^\downarrow\neq\mt_*=\mt_*^\downarrow$. 
			\item[(v)] $\mt=\mt_*=\mt^\downarrow=\mt_*^\downarrow$.
		\end{itemize}
		In the last three cases~$(\mt\cup\mt_*^\downarrow)^\downarrow=\mt\cup\mt_*^\downarrow$ holds, and therefore~$\eta_{\mt}=0$.
	\end{rem}

As a consequence of Proposition \ref{prop-steadystates} we expect convergence of the solution of the reversal dynamics to the equilibrium
\begin{equation}\label{f-infty}
  f_\infty := (1+h_\infty)\mu := \left\{ \begin{array}{ll}
     f_I  & \mbox{ on all } \mt \mbox{ with empty } \mt_* \,,\\
     (1+\eta_{\mt})\mu  & \mbox{ on all } \mt \mbox{ with non-empty } \mt_* \,,\\ 0 & \mbox{ on } \ms\setminus\mk \,,
  \end{array}\right.
\end{equation}
with 
\[
   \eta_{\mt} = \frac{\int_{\mt\cup\mt_*^\downarrow}\md f_I - \int_{\mt_*\cup\mt^\downarrow}\md f_I}{\int_{\mt\cup\mt_*^\downarrow}\md f_I + \int_{\mt_*\cup\mt^\downarrow}\md f_I} \,,\qquad \mu = \frac{f_I + f_I^\downarrow}{2}\,.
\]
By the decomposition of the dynamics, the convergence analysis can be restricted to sets of the form $\mathcal{A}= \mt\cup\mt_*\cup\mt^\downarrow\cup\mt_*^\downarrow$. We start by introducing a relative entropy as a modification of~\eqref{c4:H},
which can also be motivated by the standard form of entropies for Markov processes with integrand
\[
   \frac{(f-f_\infty)^2}{f_\infty} = \frac{(h-\eta_{\mt})^2 \mu}{1+\eta_\mt} \,,\qquad \mbox{on } \mt\,.
\]
We shall use
	\begin{align}\label{def-HTf}
		\mathcal{H}_{\mt}[f]&=\frac12\int_\mt(h-\eta_{\mt})^2\md \mu + \frac12\int_{\mt_*}(h+\eta_{\mt})^2\md \mu\\
		&=\frac12\int_\mt h^2\md \mu + \frac12\int_{\mt_*}h^2\md \mu -\frac12(\rho+\rho_*)\eta_{\mt}^2 \,,\nonumber
	\end{align}
	with $\rho$ and $\rho_*$ as in Proposition \ref{prop-steadystates}. For odd functions $h$, this quantity controls the~$L^2(\mu)$ distance between~$h|_\mathcal{A}$ and its equilibrium. 
	The derivative in time of~$\mathcal{H}_{\mt}[f]$ can be computed as previously, and we obtain
	\begin{align}
		\frac{\md}{\md t} \h_{\mt}[f] 
		&=- 2\iint_{\mt\times\mt_*}  b(x,x_*) (h+\ha)^2 \, \md \mu \,\md \mu_* =: - \D_{\mt}[f]. \label{dthT}
	\end{align}
Our goal is to identify situations where the dissipation $\D_{\mt}$ controls the entropy $\mathcal{H}_{\mt}$. By the following result, this is true for the case where the collision kernel $b$ is bounded away from zero on~$\mt\times\mt_*$.
	\begin{proposition}\label{propEstimateHT}
		With the notation from Proposition \ref{prop-steadystates} we have
		\begin{equation*}
			2\min(\rho,\rho_*)\mathcal{H}_{\mt}[f]\leqslant\iint_{\mt\times\mt_*} (h+\ha)^2 \, \md \mu \,\md \mu_* \,.
		\end{equation*}
	\end{proposition}
	\begin{proof}
		We compute the right-hand side of the desired inequality: 
		\begin{equation*}
			\iint_{\mt\times\mt_*} (h+\ha)^2 \, \md \mu \,\md \mu_*=\rho_*\int_{\mt}h^2\md \mu + \rho\int_{\mt_*}h^2\md \mu + 2 \rho\rho_*\langle h\rangle_\mt\langle h\rangle_{\mt_*}.
		\end{equation*}
		Now we expand an expression similar to the expression~\eqref{def-HTf} of~$\mathcal{H}_{\mt}[f]$, but with the weights~$\rho$ and~$\rho_*$ for the quadratic parts in~$h$:
		\begin{align*}
			\rho_*\int_\mt(h&-\eta_{\mt})^2\md \mu + \rho\int_{\mt_*}(h+\eta_{\mt})^2\md \mu\\
			&=\rho_*\int_{\mt}h^2\md \mu + \rho\int_{\mt_*}h^2\md \mu +2\rho\rho_*(\eta_{\mt}^2-\langle h\rangle_{\mt}\eta_{\mt}+\langle h\rangle_{\mt_*}\eta_{\mt})\\
			&=\rho_*\int_{\mt}h^2\md \mu + \rho\int_{\mt_*}h^2\md \mu +2\rho\rho_*(\eta_{\mt}-\langle h\rangle_{\mt})(\eta_{\mt}+\langle h\rangle_{\mt_*}) +2\rho\rho_*\langle h\rangle_\mt\langle h\rangle_{\mt_*}\\
			&=\iint_{\mt\times\mt_*} (h+\ha)^2 \, \md \mu \,\md \mu_* - 2\Big(\frac{\rho\rho_*}{\rho+\rho_*}\Big)^2(\langle h\rangle_{\mt}+\langle h\rangle_{\mt_*})^2,
		\end{align*}
		where we have used~$\eta_{\mt}-\langle h\rangle_\mt=-\frac{\rho_*}{\rho+\rho_*}(\langle h\rangle_{\mt}+\langle h\rangle_{\mt_*})$ and~$\eta_{\mt}+\langle h\rangle_{\mt_*}=\frac{\rho}{\rho+\rho_*}(\langle h\rangle_{\mt}+\langle h\rangle_{\mt_*})$.
		We then obtain
		\begin{equation*}
			\iint_{\mt\times\mt_*} (h+\ha)^2 \, \md \mu \,\md \mu_*\geqslant\rho_*\int_\mt(h-\eta_{\mt})^2\md \mu + \rho\int_{\mt_*}(h+\eta_{\mt})^2\md \mu\geqslant2\min(\rho,\rho_*)\mathcal{H}_{\mt}[f] \,,
		\end{equation*}
		and this ends the proof.
	\end{proof}
	
	It will turn out that the condition of boundedness away from zero of $b$ can be removed under the additional assumption that $\mt$ and $\mt_*$ are compact. An important tool will be a strengthening of the connectivity of $\mt$. As a preparation we note that the graph $\Gamma$, when restricted to a pair $(\mt,\mt_*)$ can also be seen as a connected \emph{bipartite} graph with edges $(x,x_*)$ in~$\mt\times\mt_*$ consisting of pairs of collision partners, i.e. $b(x,x_*)>0$. 
	
	\begin{Def} \label{defCoveringsGraph} Let $\mt$ be a connected component of $\Gamma$, let $\beta>0$, and let
	\[
	\mathcal{C}= \{T_1,\ldots,T_n,T_{*,1},\ldots,T_{*,n_*}\}\] 
	be a finite set of measurable sets such that
	\begin{align*}
	    \mt &= \bigcup_{i=1}^n T_i \,,\qquad 
	    \rho_i := \int_{T_i}\md\mu > 0 \,,\quad i=1,\ldots, n\,,\\
	   \mt_* &= \bigcup_{j=1}^{n_*} T_{*,j} \,,  \qquad 
	   \rho_{*,j} := \int_{T_{*,j}}\md\mu > 0 \,,\quad j=1,\ldots, n_*\,.
	\end{align*}
	We say that $T_i, T_{*,j}\in\mathcal{C}$ are \underline{$\beta$-linked}, iff
	\begin{equation*} 
		  b(x,x_*)\geqslant\beta \quad\text{ for all }(x,x_*)\in T_i\times T_{*,j} \,.
	\end{equation*}
	If the bipartite graph with the node set $\mathcal{C}$ and edges between $\beta$-linked nodes is connected, we call $\mathcal{C}$ a \underline{finite $\beta$-connected covering} of $\mt\cup\mt_*$.	
	\end{Def}
	
	A finite $\beta$-connected covering can be seen as a strengthened version of the bipartite graph mentioned in Definition \ref{defGraphGamma}. The following result (together with the previous lemma) shows that its existence implies the desired control of~$\mathcal{H}_\mt[f]$ by~$\mathcal{D}_\mt[f]$.
	
	\begin{lemma}\label{lemmaMainEstimate}
	Let $\mt$ be a connected component of $\Gamma$, let $\mt_*$ be non-empty, and let $\beta>0$. Assume there exists a finite $\beta$-connected covering $\mathcal{C}$ of $\mt\cup\mt_*$. Then there exists a constant~$C>0$ (only depending on~$\mu$,~$\beta$, and $\mathcal{C}$) such that
		\begin{equation*}
			\iint_{\mt\times\mt_*} (h+\ha)^2 \, \md \mu \,\md \mu_*\leqslant C\iint_{\mt\times\mt_*} b(x,x_*)(h+\ha)^2 \, \md \mu \,\md \mu_*.
		\end{equation*}
	\end{lemma}
	\begin{proof}
		We have 
		\begin{equation}
			\iint_{\mt\times\mt_*} (h+\ha)^2 \, \md \mu \,\md \mu_*\leqslant\sum_{i=1}^n \sum_{j=1}^{n_*}\iint_{T_i\times T_{*,j}} (h+\ha)^2 \, \md \mu \,\md \mu_* \,.
			\label{estimate_sumsij}
		\end{equation}
	Concentrating on one pair $(i,j)$, there exists a path of $\beta$-links connecting $T_i$ and $T_{*,j}$, given by the index sequence~$(i=i_0,j_0,i_1,j_1,\dots,i_k,j_k=j)$. With~$x_{i_\ell}\in T_{i_\ell}$ and~$x_{*,j_\ell}$ in~$T_{*,j_\ell}$ for all~$0\leqslant\ell\leqslant k$ we have
		\begin{align*}
			(h(x_{i_0})&+h(x_{*,j_k}))^2=\left(\sum_{\ell=0}^k(h(x_{i_\ell})+h(x_{*,j_\ell}))-\sum_{\ell=0}^{k-1}(h(x_{i_{\ell+1}})+h(x_{*,j_\ell}))\right)^2\\
			&\leqslant(2k+1)\left(\sum_{\ell=0}^k(h(x_{i_\ell})+h(x_{*,j_\ell}))^2+\sum_{\ell=0}^{k-1}(h(x_{i_{\ell+1}})+h(x_{*,j_\ell}))^2\right)\\
			&\leqslant\frac{2k+1}{\beta}\left(\sum_{\ell=0}^kb(x_{i_\ell},x_{*,j_\ell})(h(x_{i_\ell})+h(x_{*,j_\ell}))^2+\sum_{\ell=0}^{k-1}b(x_{i_{\ell+1}},x_{*,j_\ell})(h(x_{i_{\ell+1}})+h(x_{*,j_\ell}))^2\right) \,.
		\end{align*}
		Now we integrate against~$\md \mu(x_{i_0})\cdots \md\mu(x_{*,j_k})$ over~$T_{i_0}\times\hdots\times T_{*,j_k}$ and obtain
		(after division by $\rho_{i_0}\cdots\rho_{*,j_k}$)
		\begin{align*}
			&\frac{1}{\rho_{i_0}\rho_{*,j_k}}\iint_{T_i\times T_{*,j}}(h+h_*)^2\, \md \mu \,\md \mu_*\\
			&\hspace{1em}\leqslant\frac{2k+1}{\beta}\Biggl(\sum_{\ell=0}^k\frac{1}{\rho_{i_\ell}\rho_{*,j_\ell}}\iint_{T_{i_\ell}\times T_{*,j_\ell}}\hspace{-1em}b(x,x_*)(h+h_*)^2\, \md \mu \,\md \mu_* \\
			&\hspace{2em}+\sum_{\ell=0}^{k-1}\frac{1}{\rho_{i_{\ell+1}}\rho_{*,j_\ell}}\iint_{T_{i_{\ell+1}}\times T_{*,j_\ell}}\hspace{-1em}b(x,x_*)(h+h_*)^2\, \md \mu \,\md \mu_*\Biggr)\,,
		\end{align*}
Since~$T_{i_\ell}\subset\mt$ and~$T_{*,j_\ell}\subset\mt_*$ for all~$0\leqslant\ell\leqslant k$, we finally obtain
		\begin{equation*}
			\begin{split}
			&\iint_{T_{i}\times T_{*,j}}(h+h_*)^2\, \md \mu \,\md \mu_* \\
			&\hspace{1em} \leqslant\frac{2k+1}{\beta}\left(\sum_{\ell=0}^k\frac{\rho_{i_0}\rho_{*,j_k}}{\rho_{i_\ell}\rho_{*,j_\ell}}+\sum_{\ell=0}^{k-1}\frac{\rho_{i_0}\rho_{*,j_k}}{\rho_{i_{\ell+1}}\rho_{*,j_\ell}}\right)\iint_{\mt\times\mt_*}(x,x_*)(h+h_*)^2\, \md \mu \,\md \mu_* \,.  
			\end{split}
		\end{equation*}
		Summation with respect to $i$ and $j$ and using~\eqref{estimate_sumsij} completes the proof.
	\end{proof}
	
	Finally, it remains to provide sufficient conditions for the existence of a finite $\beta$-connected covering.
	
	\begin{proposition} \label{propExistCoverings} 
	Let $\mt$ be a connected component of $\Gamma$, let $\mt_*$ be non-empty, let both sets be compact, and let $b$ be lower semicontinuous. Then there exists~$\beta>0$ and a finite $\beta$-connected covering of $\mt\cup\mt_*$.
	\end{proposition}
	\begin{proof}
	We work with the connected bipartite graph with the node set $\mt\cup\mt_*$ and edges between collision partners $x\in\mt$
	and $x_*\in\mt_*$. For $(x,y)\in\mt\times\mt_*$ we denote by~$P(x,y)$ the set of connecting paths of the form~$p=(x=x_0,y_0,x_1,\ldots,x_k,y_k=y)$. For such a path~$p$ we denote
	\[
	  \beta_p := \min\big(\{b(x_i,y_i), 0\leqslant i\leqslant k\}\cup\{b(x_{i+1},y_i), 0\leqslant i\leqslant k-1\} \big)\,,
	\]
	and we define
	\[
	  \bar\beta(x,y) := \sup\{\beta_p:\, p\in P(x,y)\} \,.
	\]
	By connectedness we have $\bar\beta(x,y)>0$, $(x,y)\in\mt\times\mt_*$.
	
		Let us prove that~$(x,y)\mapsto\bar{\beta}(x,y)$ is lower semicontinuous on~$\mt\times\mt_*$. Indeed, if~$\varepsilon>0$ there exists a path $p\in P(x,y)$ such that $\beta_p > \bar\beta(x,y)-\varepsilon$. In particular, since~$b$ is lower semicontinuous, there exists~$r>0$ such that for all~$\widetilde{x}\in B(x,r)\cap\mt$ and for all~$\widetilde{y}\in B(y,r)\cap\mt_*$, we have~$b(\widetilde{x},y_0)>\bar{\beta}(x,y)-\varepsilon$ and~$b(x_k,\widetilde{y})>\bar{\beta}(x,y)-\varepsilon$. Therefore with the path~$(\widetilde{x},y_0,\dots,x_k,\widetilde{y})$, we obtain that~$\bar{\beta}(\widetilde{x},\widetilde{y})\geqslant\bar{\beta}(x,y)-\varepsilon$. Therefore~$\bar{\beta}$ is lower semicontinuous on~$\mt\times\mt_*$ and reaches there its minimum~$\bar{\beta}_0>0$. 
		
		We now fix~$0<\beta<\bar{\beta}_0$, and for any~$x,y\in\mt\times\mt_*$ we define~$\bar{\delta}(x,y)$ the supremum of possible~$0<\delta\leqslant\mathrm{diam}(\mk)$ such that there exists a path~$p\in P(x,y)$ for which the property as above is valid even when moving all the points by less than~$\delta$: if~$p=(x=x_0,y_0,x_1,y_1,\dots,x_k,y_k=y)$ then for all $(\widetilde{x}_0,\widetilde{y}_0,\dots,\widetilde{x}_k,\widetilde{y}_k)$ such that~$\widetilde{x}_i\in B(x_i,\delta)\cap\mt$ and~$\widetilde{y}_i\in B(y_i,\delta)\cap\mt_*$ we have the estimation~$b(\widetilde{x}_i,\widetilde{y}_i)>\beta$ for all~$0\leqslant i\leqslant k$ and~$b(\widetilde{y}_i,\widetilde{x}_{i+1})>\beta$ for all~$0\leqslant i<k$. By lower semicontinuity of~$b$, such a positive~$\delta$ exists for all~$x,y$ in~$\mt\times\mt_*$. 
		
		Let us now prove that~$\bar{\delta}$ is lower semicontinuous. If for a given~$\delta>0$ we have a path as described above between~$x\in\mt$ and~$y\in\mt_*$, and~$d(\widetilde{x},x)<\varepsilon$ and~$d(\widetilde{y},y)<\varepsilon$ for~$\varepsilon<\delta$, then~$B(\widetilde{x},\delta-\varepsilon)\subset B(x,\delta)$ and~$B(\widetilde{y},\delta-\varepsilon)\subset B(y,\delta)$, therefore taking the same path, we have the same property with radius~$\delta-\varepsilon$. This shows that~$\bar{\delta}(\widetilde{x},\widetilde{y})\geqslant\bar{\delta}(x,y)-\varepsilon$. Therefore~$\bar{\delta}$ is lower semicontinuous and reaches its minimum~$\bar{\delta}_0>0$ on~$\mt\times\mt_*$. So if we fix now~$\delta<\bar{\delta}_0$, we have the property of existence of a path as above, uniformly in~$x$ and~$y$.
		
		We now take a finite covering of the compact set~$\mt$ (resp.~$\mt_*$) by open sets~$T_i$ (resp.~$T_{*,j}$) of diameter less than~$\delta$. By Proposition~\ref{prop-steadystates}, since~$\mt$ and~$\mt_*$ are open in~$\mk$, without loss of generality (replacing~$T_i$ by~$T_i\cap\mt$ for instance) we can assume that the open sets~$T_i$ (resp.~$T_{*,j}$) are subsets of~$\mt$ (resp.~$\mt_*$), to be in the framework of Definition~\ref{defCoveringsGraph}.
		
		It remains to show that~$\mathcal{C}:= \{T_i\}\cup\{T_{*,j}\}$ is $\beta$-connected. It is sufficient to show that two arbitrary vertices of different type~$T_i$ and~$T_{*,j}$ with~$\rho_i>0$ and~$\rho_{*,j}>0$ are connected (since~$\mt$ and~$\mt_*$ are not empty, we have~$\rho,\rho_*>0$ by Proposition~\ref{prop-steadystates}, so there is at least one vertex of each type with~$\rho_i>0$ and~$\rho_{*,j}>0$). We take~$x\in T_i$ and~$y\in T_{*,j}$ and we obtain a path~$(x=x_0,y_0,x_1,y_1,\dots,x_k,y_k=y)$ as previously, such that 
		for all $(\widetilde{x}_0,\widetilde{y}_0,\dots,\widetilde{x}_k,\widetilde{y}_k)$ such that~$\widetilde{x}_\ell\in B(x_\ell,\delta)\cap\mt$ and~$\widetilde{y}_\ell\in B(y_\ell,\delta)\cap\mt_*$ we have the estimate~$b(\widetilde{x}_\ell,\widetilde{y}_\ell)>\beta$ for all~$0\leqslant\ell\leqslant k$ and~$b(\widetilde{y}_\ell,\widetilde{x}_{\ell+1})>\beta$ for all~$0\leqslant\ell<k$. We use the coverings to find~$i_\ell$ and~$j_\ell$ for~$0\leqslant\ell\leqslant k$ such that~$x_\ell\in T_{i_\ell}$ and~$y_\ell\in T_{*,j_\ell}$ for all~$0\leqslant\ell\leqslant k$ (with~$i_0=i$ and~$j_k=j$). Since~$T_{i_\ell}$ is open in~$\mk$, it follows that~$\rho_{i_\ell}>0$ by definition of the support~$\mk$ of~$\mu$. Similarly~$\rho_{*,j_\ell}>0$.
		Now since the diameter of~$T_{i_\ell}$ (resp.~$T_{*,j_\ell}$) is less than~$\delta$, we obtain that~$T_{i_\ell}\subset B(x_\ell,\delta)$ (resp.~$T_{*,j_\ell}\subset B(y_\ell,\delta)$), and therefore we obtain that for all~$\widetilde{x}_\ell\in T_{i_\ell}$ and~$\widetilde{y}_\ell\in T_{*,j_\ell}$, we have~$b(\widetilde{x}_\ell,\widetilde{y}_\ell)>\beta$ which means that~$T_{i_\ell}$ and~$T_{*,j_\ell}$ are~$\beta$-linked. Similarly, we have that~$T_{*,j_\ell}$ and~$T_{i_{\ell+1}}$ are~$\beta$-linked for~$0\leqslant\ell<k$. Therefore we have a path in the graph~$\mathcal{C}$ from~$T_i=T_{i_0}$ to~$T_{*,j}=T_{*,j_k}$, and this ends the proof.
	\end{proof}
	
	Finally, combining Lemma~\ref{lemmaMainEstimate} and Propositions~\ref{propEstimateHT} and~\ref{propExistCoverings} with the expression~\eqref{dthT} of the dissipation of~$\mathcal{H}_{\mt}[f]$, we obtain, when the connected components~$\mt$ and~$\mt_*$ are compact and not empty or, alternatively, when $b$ is bounded away from zero on~$\mt\times\mt_*$, that there exists a constant~$\lambda>0$ (only depending on~$\mu$,~$b$ and~$\mathcal{T}$, and not on the initial condition~$h_I$) such that
	\begin{equation}\label{HT-diss}
		\frac{\md}{\md t}\mathcal{H}_{\mt}[f]\leqslant-\lambda\mathcal{H}_{\mt}[f].
	\end{equation}
	Therefore we have exponential decay on~$\mathcal{A}$ of the solution towards the steady-state on~$\mathcal{A}$ (given by~$(1+\eta_\mt)\mu$ on~$\mt\cup\mt_*^{\downarrow}$ and~$(1-\eta_\mt)\mu$ on~$\mt_*\cup\mt^\downarrow$. The main theorem of this section summarizes these results and provides a case where the number of connected components is finite (and they are all compact), and therefore we have exponential convergence to the steady state.
	
	\begin{theorem} \label{thm-asymptotic-behaviour}
Let~$f_I$ be a probability measure on~$\ms$ and~$\mu=\frac12(f_I+f_I^\downarrow)$ with support $\mk$. Let the collision kernel~$b$ be lower semicontinuous and bounded, let the graph~$\Gamma$ be as in Definition~\ref{defGraphGamma}, and let~$\mk=\mk_*$ (i.e., every element of~$\mk$ has a collision partner in~$\mk$). \\
Then~$\Gamma$ has a finite number of connected components, which are all compact and the solution~$f(t)$ of the reversal collision dynamics given by Theorem~\ref{thm-existence-uniqueness} converges as $t\to\infty$ exponentially in total variation distance to $f_\infty$, given by \eqref{f-infty}, i.e., there exist $C,\lambda>0$, such that
\[
  \sup_{A\in\sigma(\ms)} \left| \int_A \md f(t) - \int_A \md f_\infty\right| \leqslant C e^{-\lambda t} \,,\qquad t\geqslant 0 \,.
\]
\end{theorem}
	
\begin{proof} 
The fact that every point in $\mk$ has a collision partner excludes case (i) of Proposition~\ref{prop-decomposition}. This implies that $\mk$ is the disjoint union of open sets $\mt$ in $\mk$, the connected components of~$\Gamma$. By compactness of $\mk$, this covering is finite, and all components are closed, and therefore compact in $\ms$.

For estimating the total variation distance we observe that for $A\in\sigma(\ms)$ we have
\begin{equation*}
 \left(\int_A \md f(t) - \int_A \md f_\infty\right)^2 =
    \left(\int_A (h(t)-h_\infty)\md\mu\right)^2 \leqslant
    \int_{\mk} (h(t)-h_\infty)^2 \md\mu = 
    \sum_{\mt} \int_{\mt} (h(t)-\eta_{\mt})^2 \md\mu
\end{equation*}
In the finite sum on the right hand side we introduce a grouping into pairs of the form $(\mt,\mt_*)$, whence the terms in the sum take the form $\mathcal{H}_{\mt}[f]$ (since $\eta_{\mt_*} = -\eta_{\mt}$). Lemma~\ref{lemmaMainEstimate} and Propositions~\ref{propEstimateHT} and~\ref{propExistCoverings} imply a differential inequality of the form \eqref{HT-diss} for each $\mathcal{H}_{\mt}[f]$ and therefore its exponential
decay. 
	\end{proof}
	
The following corollary shows that Theorem \ref{thm-asymptotic-behaviour} covers the case of our motivating example~\eqref{c4:REVT} on the circle.
	\begin{cor}\label{corollary-main-theorem}
		Let the assumption of Theorem \ref{thm-asymptotic-behaviour} hold, but with $\mk=\mk_*$ replaced by~$b(x,x^\downarrow)>0$ for all~$x\in\ms$. Then the conclusions of Theorem \ref{thm-asymptotic-behaviour} are valid. Furthermore each connected component~$\mt$ of $\Gamma$ satisfies $\mt_*=\mt^\downarrow$, and therefore we have either Case (ii) or Case (v)
		of Remark \ref{rem-five-cases}.
	\end{cor}
	\begin{proof} Since~$\mu$ is symmetric, $\mk$ is symmetric. So if~$x\in\mk$,~$x^\downarrow$ is a collision partner of~$x$ which belongs to~$\mk$. Therefore~$\mk=\mk_*$ and the conclusions of Theorem~\ref{thm-asymptotic-behaviour} hold. We also have~$\mt^\downarrow\subset\mt_*$ (and thus~$\mt^\downarrow=\mt_*$ since both are connected component of the graph~$\Gamma$) and the only possibilities from Remark~\ref{rem-five-cases} are (ii) and~(v).
	\end{proof}
	
\begin{rem}
An example, showing that without the condition $\mk=\mk_*$, the number of connected components may be infinite: let $\ms = [-1,1]$ with~$x^\downarrow=-x$, 
	\[
	   \mk = \mk_+ \cup \mk_+^\downarrow \qquad\mbox{with }
	   \mk_+ = \{0\} \cup \bigcup_{k=1}^\infty \mt_k \,,\quad
	   \mt_k = \left[ \frac{1}{2k+1}, \frac{1}{2k}\right] \,,
	\]
	and 
	\[
	  b(x,x_*)>0 \qquad\mbox{iff } x<x_*+x_*^2 \mbox{ and }
	  x_* < x + x^2 \qquad \mbox{on } [0,1]^2 \,,
	\]
	with $b=0$ on~$[-1,0]\times[0,1]$, and~$b$ defined on~$[-1,1]\times[-1,0]$ by symmetry~\eqref{symmetry-b}.
	One can show that $\mt_{k,*} = \mt_k$ and that each $\mt_k$ is a
	connected component of $\Gamma$. Note that $0\in\mk$
	does not have a collision partner.\\
	The essential thing to show is that no element in~$\mt_k$ has a collision partner in $\mt_{k+l}$ for $l\geqslant 1$. This would require
	\[
	   \frac{1}{2k+1} < \frac{1}{2(k+l)} + \frac{1}{4(k+l)^2} \,.
	\]
	It is an easy computation to show that this can never hold.\\
	For each $k\geqslant 1$ Case (iii) of Remark \ref{rem-five-cases} applies, and the solution of the reversal collision dynamics converges to $\mu$ on $\mt_k$.
	By the compactness of $\mt_k$ the convergence is exponential, but with a $k$-dependent rate, which might degenerate as $k\to\infty$, depending on the choice of $b$.
	\end{rem}

\begin{rem}
Another example shows that without the condition $\mk=\mk_*$, the connected components, even in finite number, may not be compact. Again with $\ms = [-1,1]$ and~$x^\downarrow=-x$, we set~$\mk=\ms$, and
\[
	  b(x,x_*)=\begin{cases}|x-x_*|-1 &\mbox{if }|x-x_*|\geqslant1,\\
	  0 & \mbox{if }|x-x_*|\leqslant1.\end{cases}
	  \]

	One can show here that the connected components of~$\Gamma$ are~$[-1,0)$, $\{0\}$ and~$(0,1]$. Again,~$0$ does not have a collision partner, and the two other components are now in Case (ii) of Remark~\ref{rem-five-cases}. But we cannot expect an exponential rate of convergence, since the rate of interaction with particles located close to~$0$ degenerates.
	\end{rem}

	\begin{rem}
		We shall provide a last example, showing that the convergence rate $\lambda>0$ in \eqref{HT-diss} is not only depending on $\mathcal{K}$ and on the kernel $b$, but also on the invariant measure $\mu$. 
		
		We consider the motivating example \eqref{c4:REVT} with
		\[
		b(\vp,\vpa) = \left\{\begin{array}{ll} 1, & \text{for } d(\vp,\vpa)> \pi/2 \,,\\
		0, & \text{for } d(\vp,\vpa)\leqslant \pi/2 \,,\end{array}\right.
		\]
		and with the initial conditions 
		\[
			f_I = 2\alpha \delta_0 + 2\beta \delta_{2\pi/3} + 2\gamma \delta_{-2\pi/3} \,, \qquad \alpha+\beta+\gamma = \frac{1}{2} \,.
		\]
		The invariant measure therefore is given by
		\[
			\mu = \alpha \left(\delta_0 + \delta_{\pi}\right) + \beta \left(\delta_{2\pi/3}+\delta_{-\pi/3}\right)  + \gamma \left(\delta_{-2\pi/3}+\delta_{\pi/3}\right).
		\]
		We note that we are in case (v) of Remark \ref{rem-five-cases} with $\mk=\mathcal{T}=\mathcal{T}_*=\mathcal{T}^\downarrow=\mathcal{T}_*^\downarrow$. Due to the oddness of $h$ it is enough to investigate the dynamics on $\supp{f_I}$. Thus, we define by $h_{0}(t), h_{2\pi/3}(t), h_{-2\pi/3}(t)$ the coefficients of $h$ at the points $0, \frac{2\pi}{3}, -\frac{2\pi}{3}$. From \eqref{lin-eq-h} we obtain the linear system
		\begin{equation*}
			\frac{\md }{\md t} 
			\begin{pmatrix}
				h_0(t) \\
				h_{2\pi/3}(t) \\
				h_{-2\pi/3}(t)
			\end{pmatrix}
		= -2 
		\begin{pmatrix}
			\beta+\gamma & \beta & \gamma \\
			\alpha &\alpha+\gamma & \gamma \\
		    \alpha & \beta & \alpha+\beta
		\end{pmatrix}
		\begin{pmatrix}
			h_0(t) \\
			h_{2\pi/3}(t) \\
			h_{-2\pi/3}(t)
		\end{pmatrix}\,,
		\end{equation*}
	where the characteristic polynomial $p(\xi)$ of the coefficient matrix is given by 
	\[
		p(\xi) = \xi^3 + 2\xi^2 + 4\xi \left(\alpha^2 + \beta^2 +\gamma^2 + 2 (\alpha\beta+\alpha\gamma+\beta \gamma)\right) + 32 \alpha\beta\gamma \,.
	\]
	In the case where $\alpha \ll \beta, \gamma$, we get the following asymptotic behaviour for the eigenvalues:
	\begin{equation*}
		\xi_1 \approx -32 \alpha \beta \gamma \,,\qquad 
		\xi_{2,3} \approx -1 \,.
	\end{equation*}
	Hence also the convergence rate $\lambda$ in \eqref{HT-diss} degenerates as $\alpha \to 0$. 
	\end{rem}
	
	\section{Reversal collisions on the torus $\T$}\label{section-torus}
	
	This section is devoted to a slight generalization of the motivating example \eqref{c4:REVT} with $\ms = \T$ and with
	\begin{equation*}
		b(\vp,\vpa) = \begin{cases}
			1 & \text{if }\,\, d(\vp,\vpa)>\pi-\alpha, \\
			0 & \text{if }\,\, d(\vp,\vpa)\leqslant \pi - \alpha,
		\end{cases} \qquad 0<\alpha<\pi \,.
	\end{equation*}  
	The governing equation becomes 
	\begin{equation*}
		\pa_t f =  \int_{d(\vp,\vpa) >\pi-\alpha} \left(\fd \fad -f\fa \right) d\vpa \,,
	\end{equation*}
	equipped with initial conditions 
	\[
		f(\vp,0)=f_I(\vp) \,, \qquad \text{for } \vp \in \T \,,
	\]
	where~$f_I$ is a probability measure on~$\T$. The case $\alpha=\pi/2$ corresponds to the motivating example \eqref{c4:REVT}, discussed and simulated in the following Section \ref{c4:numerics}. 
	
	The immediate observation that due to $d(\vp,\vp^\downarrow)=\pi>\pi-\alpha$ for any $\vp \in \mk$ its opposite $\vp^\downarrow$ has to be a collision partner, implies $\mt^\downarrow \subset \mt_*$ for all connected components $\mt$ of $\Gamma$. Consequently, since $\mt^\downarrow$ and $\mt_*$ are also connected components,
	\[
		\mt^\downarrow = \mt_* \,.
	\]
	Thus, we can exclude the cases (i), (iii) and (iv) of Remark \ref{rem-five-cases}.
	
	Another crucial observation is that any pair $\vp, \psi \in \mk$ satisfies
	\begin{align}\label{torus_adj}
		d(\vp,\psi) < \alpha \quad \Rightarrow\quad \vp \longleftrightarrow \psi \,,
	\end{align}
	which can be seen easily since 
	\[
	\pi = d(\vp,\vp^\downarrow) = d(\vp,\psi)+d(\psi,\vp^\downarrow) < \alpha + d(\psi,\vp^\downarrow),
	\]
	which implies $d(\psi,\vp^\downarrow)>\pi-\alpha$. Therefore $\vp^\downarrow$ is a collision partner for both  $\vp$ and $\psi$.
	
	The property \eqref{torus_adj} allows us to characterize all the possible configurations. 
	
	\begin{proposition}\label{prop-configurations-torus}
	There is more than one connected component of~$\Gamma$ if and only if there exists an interval of the form~$(\psi,\psi+\alpha)$ entirely included in~$\ms\setminus\mk$. If this is the case, the connected components correspond to Case~(ii) of Remark~\ref{rem-five-cases}, their number is even and at most~$2\lfloor \pi/\alpha\rfloor$.
	
	If no such interval exists, there is only one connected component, corresponding to Case~(v) of Remark~\ref{rem-five-cases}.
	\end{proposition}
	
	\begin{proof}
	We first prove that if~$\mt$ is a connected component in Case~(v) of Remark~\ref{rem-five-cases}, then there is no such interval of the form~$(\psi,\psi+\alpha)$ entirely included in~$\ms\setminus\mk$.
	Indeed, for~$\vp\in\mt$, we have~$\vp^\downarrow \in\mt$, so we can pick a connecting path~$p=(\vp=\vp_0,\vp_1,\dots,\vp_n=\vp_0^\downarrow=\vp^\downarrow)$. Then for all~$0\leqslant k<n$ we have a common collision partner~$\psi_k$ of~$\vp_k$ and~$\vp_{k+1}$, that is to say~$d(\vp_k,\psi_k)>\pi-\alpha$, i.e.~$d(\vp_k,\psi_k^\downarrow)<\alpha$ (and similarly~$d(\vp_{k+1},\psi_k^\downarrow)<\alpha$). Thus, the sequence of points~$(\vp_0,\psi_0^\downarrow,\vp_1,\psi_1^\downarrow,\dots,\vp_n)$ is such that two consecutive points are at distance less than~$\alpha$, covering a half-circle  from~$\vp$ to~$\vp^\downarrow$ and, hence, excluding the occurrence of such an interval. Moreover, the sequence of points~$(\vp_0^\downarrow,\psi_0,\vp_1^\downarrow,\psi_1,\dots,\vp_n^\downarrow)$ covers the other half-circle and, thanks to the property~\eqref{torus_adj}, this implies that $\Gamma$ is connected and $\mt=\mt_*=\mt^\downarrow=\mt_*^\downarrow = \mk$. 
	
	Conversely, if~$\Gamma$ is connected, it is obviously in Case~(v) of remark~\ref{rem-five-cases}.
	If there are at least two connected components, they are therefore all in Case~(ii) of remark~\ref{rem-five-cases}. The property~\eqref{torus_adj} implies that between two different connected components there has to be a margin of width not less than $\alpha$. Between $\vp\in\mk$ and $\vp^\downarrow$ there can be at most $\lfloor \pi/\alpha\rfloor$ margins of width $\alpha$ on each half circle. This immediately implies that $\Gamma$ can have at most $2\lfloor \pi/\alpha\rfloor$ connected components. 
	\end{proof}
	
	In the limiting cases $\pi/\alpha\in\N$ the maximal number of connected components can only be reached by concentrating the mass at the $2\lfloor \pi/\alpha\rfloor$ vertices of a regular polygon, i.e.
	\[
		f_I = \sum_{i=1}^{2\pi/\alpha} \rho_i \delta_{\vp_0+i\alpha} \,, \qquad \text{with} \quad  \vp_0\in\T \,,\quad \rho_i\geqslant 0 \,.
	\]
	
	\section{Numerical simulations} \label{c4:numerics}

	This section is dedicated to illustrate the theoretical results of the previous sections with numerical simulations. We chose the setting of Section \ref{section-torus} with $\alpha = \frac{\pi}{2}$. We summarize the above considerations and the results of Proposition~\ref{prop-configurations-torus} for this special angle in the following list, characterizing the possible cases of number and properties of connected components of the graph~$\Gamma$.
	\begin{itemize}
			\item If there is no open interval of the form $(\psi,\psi+\frac{\pi}2)$ entirely included in~$\ms\setminus\mk$, then $\mk$ has only one connected component. Thus, the solution $f(t,\cdot)$ to~\eqref{c4:REVT} converges exponentially fast to the invariant measure $\mu$ as time goes to $\infty$. This corresponds to case~(v) of Remark~\ref{rem-five-cases}.
		\item Maximal four different connected components can occur. We have exactly four connected components of the graph $\Gamma$ \textit{iff} 
		\[
			\mk=\left\{\vp,\vp+\pi /2, \vp+\pi, \vp+3\pi/2\right\}, \quad  \text{for a} \quad \vp \in \T.
		\]
		In that case each of the isolated points is a connected component on which the solution~$f(t,\cdot)$ to \eqref{c4:REVT} is constant (case~(i) in Proposition~\ref{prop-decomposition}).
		\item If none of the two possibilities listed above is applicable, the graph $\Gamma$ has exactly two connected components, denoted by $\Gamma_+$ and $\Gamma_-$, on which the solution $f(t,\cdot)$ converges exponentially fast respectively to 
		\[
			2\rho_{\pm}\,\mu_{|_{\Gamma_{\pm}}}, \quad \text{as} \quad t \to \infty,
		\]
		where~$\rho_{\pm}=\int_{\Gamma_\pm}\md f_I$ (so~$\rho_++\rho_-=1$). This corresponds to Case~(ii) in Remark \ref{rem-five-cases}.
	\end{itemize}
	\paragraph{Discretization:}
	The results of the preceding section will be illustrated by numerical simulations for the simple problem \eqref{c4:REVT} on the torus with $b\equiv 1$.  Discretization is based on an equidistant grid
	\[
	\vp_k = \frac{(k-n)\pi}{n} \,,\qquad k = 0,\ldots,2n \,,
	\]
	with an even number of grid points, guaranteeing that the grid is invariant under the reversal collisions, i.e., with $\vp_k$ also $\vp_k^\downarrow=\vp_{k+n}$ is a grid point. 
	Solutions of \eqref{c4:REVT} are approximated at grid points by
	\[
	f^n(t) := (f_1(t),\ldots,f_{2n}(t)) \approx (f(\vp_1,t),\ldots,f(\vp_{2n},t)) \,,
	\]
	extended periodically by $f_{k+2n}(t) = f_k(t)$. This straightforwardly leads to the discrete model
	\begin{equation*}
		\frac{df_k}{dt} = Q_{REV}^n(f^n,f^n)_k \,,
	\end{equation*}
	with 
	\[
	Q_{REV}^n(f^n,f^n)_k := \frac{\pi}{n}\sum_{|k_*-k|>n/2} b_{k,k_*} (f_{k+n}f_{k_*+n} - f_k f_{k_*}) \,,
	\]
	and $b_{k,k_*}:=b(\vp_k,\vp_{k_*})$. 
	For the \emph{time discretization} the \emph{explicit Euler scheme} is used, such that the total mass is conserved by the discrete scheme, which has been implemented in \textsc{Matlab}. We simulated the first and third cases described at the beginning of this section. 
	
	\paragraph{The graph $\Gamma$ has one connected component:}
	Simulations have been carried out with grid-size $n=202$. Figure~\ref{c4:I15}, and the left part of Figure~\ref{c4:I12} show snapshots of the distribution function $f$ at different times together with the symmetric equilibrium $\mu$. In the left part of Figure~\ref{c4:I15decay} and the right part of~\ref{c4:I12}, the total mass~$\int_{\T} f \, d\vp =1$ as well as $\int_{-\pi}^0 f \, d\vp$ and $\int_0^{\pi} f \, d\vp$ are plotted against time. The right part of Figure~\ref{c4:I15decay} displays the $\log$-plot of $\h[f]$ belonging to the simulations of Figure~\ref{c4:I15}, which shows its exponential decay. \\
	In Figure~\ref{c4:I15} we started with asymmetric data, positive everywhere, which makes it clear that the associated graph $\Gamma$ has only one connected component and hence the solution converges to the symmetric equilibrium $\mu$. For this simulation the time-stepsize was chosen as $\Delta t=0.01$ for 1000 time-steps. \\
	Figure \ref{c4:I12} shows the evolution with positive initial conditions in the intervals~$(-3\pi/4,-\pi/4)$ and $(\pi/4,3\pi/4)$, although weighted differently. Furthermore, a perturbation in the interval~$[-\pi/4,\pi/4]$ was added, which serves as connecting point for the otherwise not connected graph. Also here convergence to the symmetric equilibrium $\mu$ can be observed. For this simulation the time-stepsize was chosen as $\Delta t=0.1$ for 5000 time-steps.
	
	\begin{figure}[h!]
		\centering
			\includegraphics[width=0.75\textwidth]{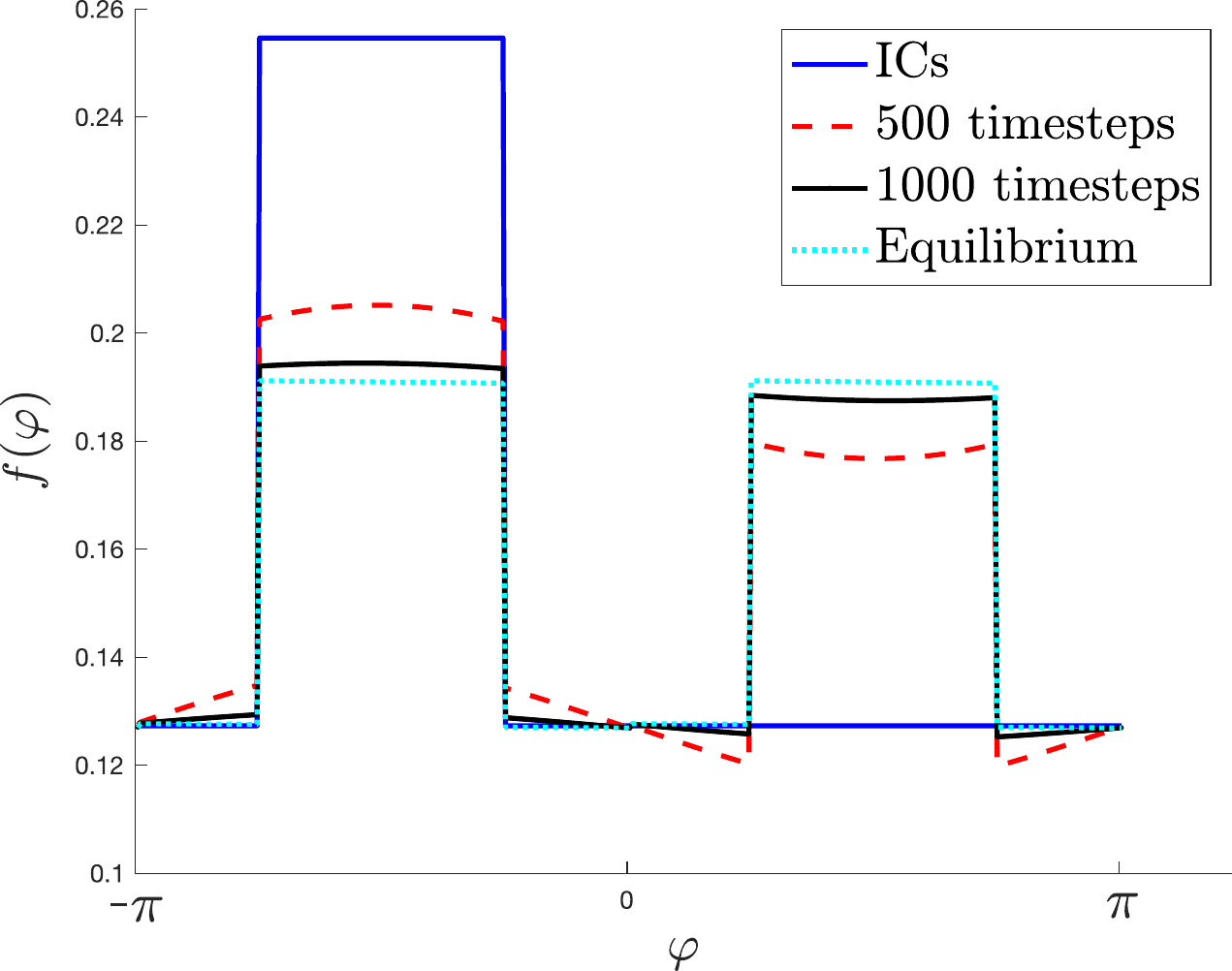}
		\caption{Initial conditions (solid dark blue) positive everywhere, $\Gamma$ has one connected component. Simulation: $f$ after 500 time-steps (dashed red), $f$ after 1000 time-steps (solid black) and $\mu$ (dotted light blue). }
		\label{c4:I15}
	\end{figure}
	
	\begin{figure}[h!]
		\begin{subfigure}{0.48\textwidth} 
			\includegraphics[width=\textwidth]{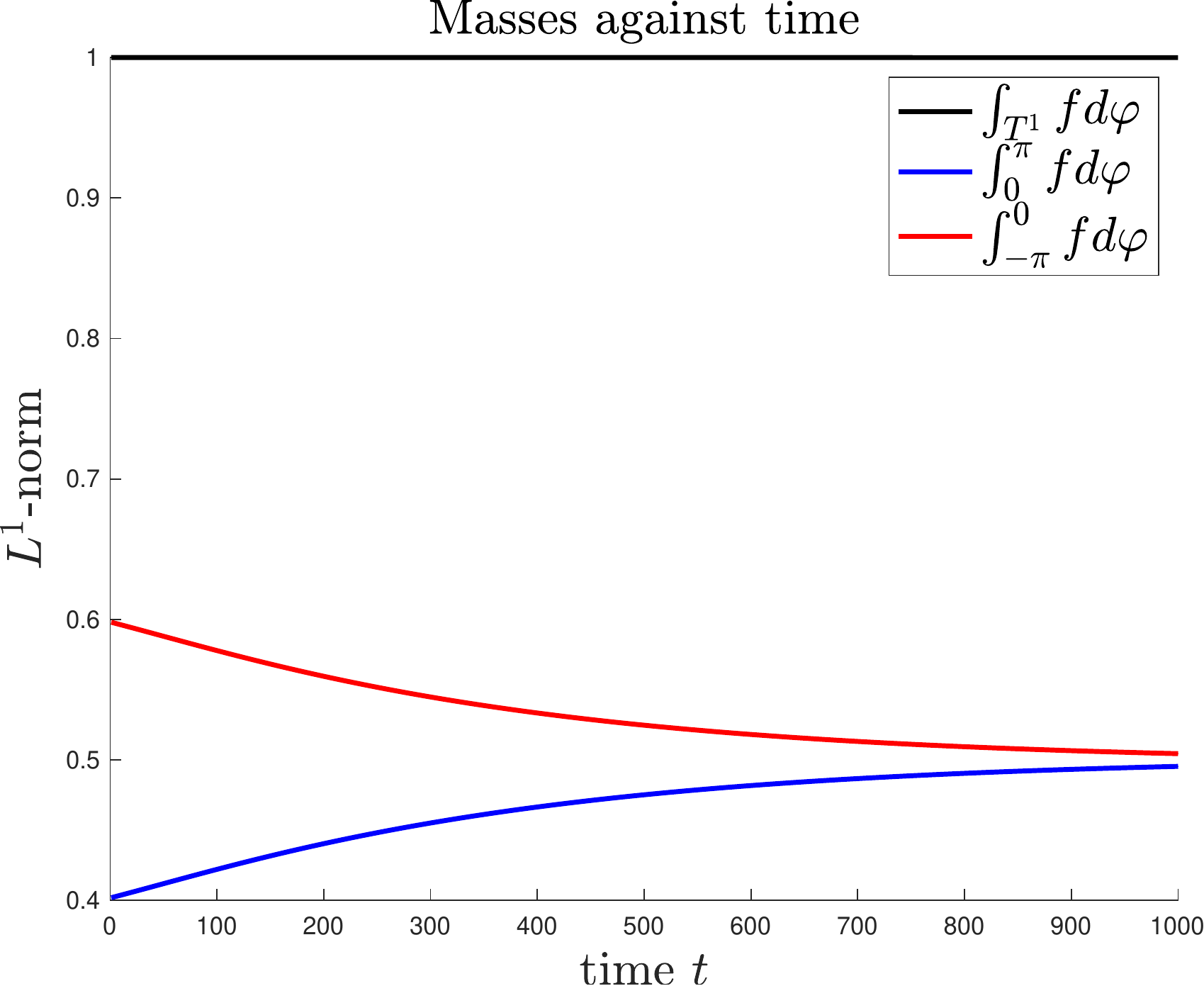}
		\end{subfigure}
		\hspace{1em}
		\begin{subfigure}{0.48\textwidth} 
			\includegraphics[width=\textwidth]{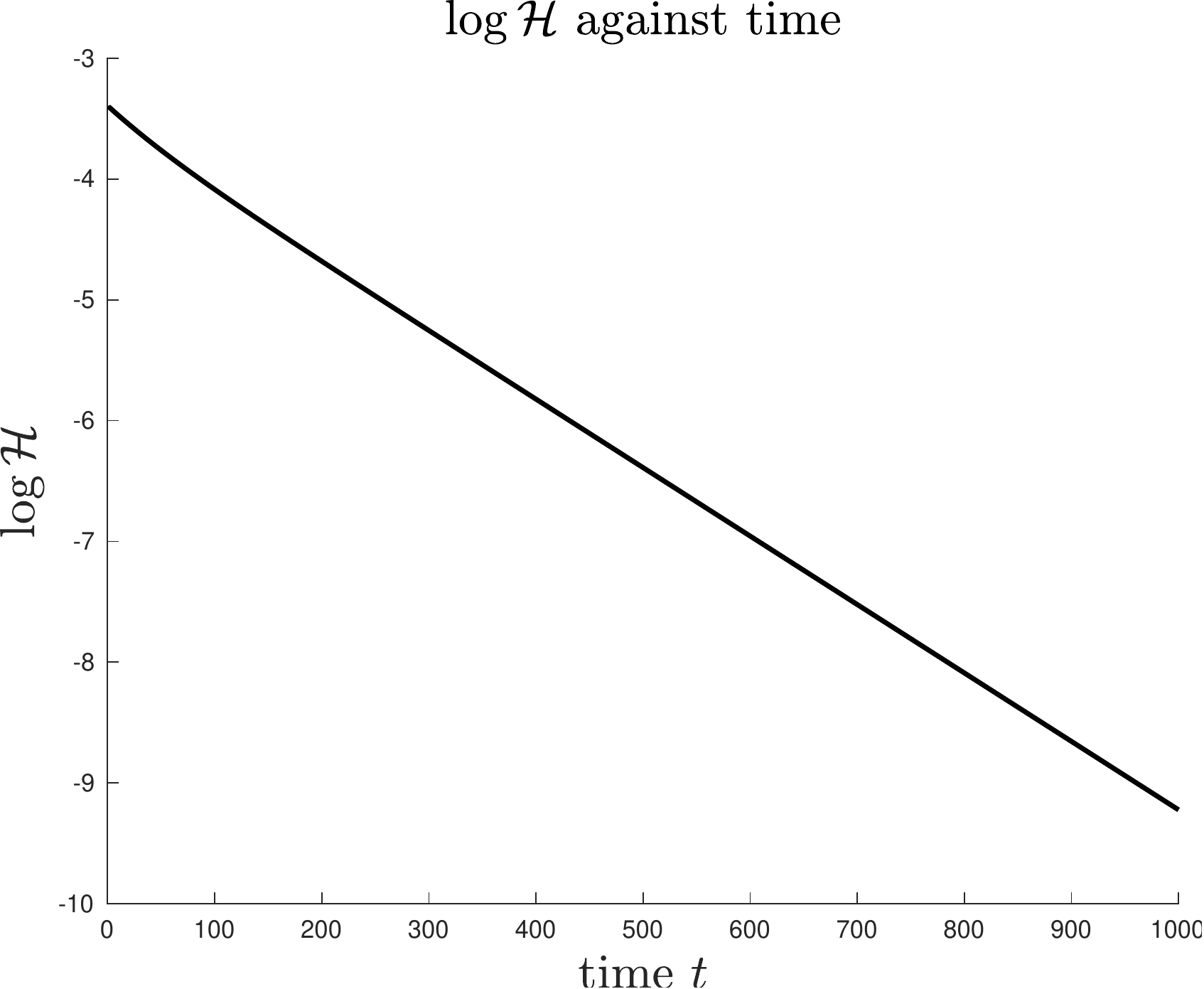}
		\end{subfigure}
		\caption{Time evolution of quantities corresponding to the simulations of Figure \ref{c4:I15}. \emph{Left:} Total mass conservation (black), masses of the positive (dark blue) and negative (red) part of the torus are different initially, but converge to the same value. \emph{Right:} $\log$-plot of the Lyapunov functional $\h$, showing exponential decay.}
		\label{c4:I15decay}
	\end{figure}
	
	\begin{figure}[h!]
		\centering
		\begin{subfigure}{0.48\textwidth} 
			\includegraphics[width=\textwidth]{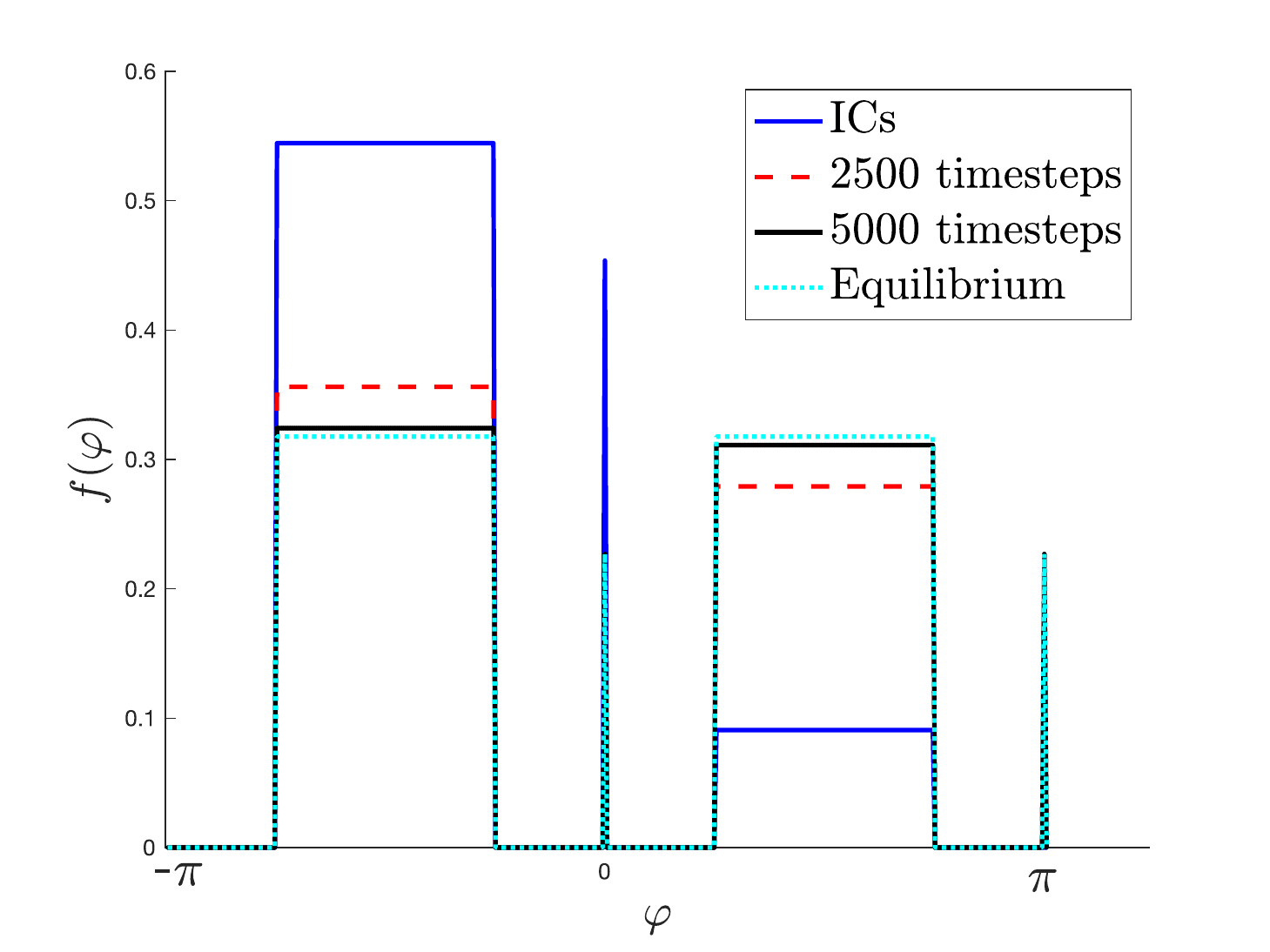}
		\end{subfigure}
		\hspace{1em}
		\begin{subfigure}{0.48\textwidth} 
			\includegraphics[width=\textwidth]{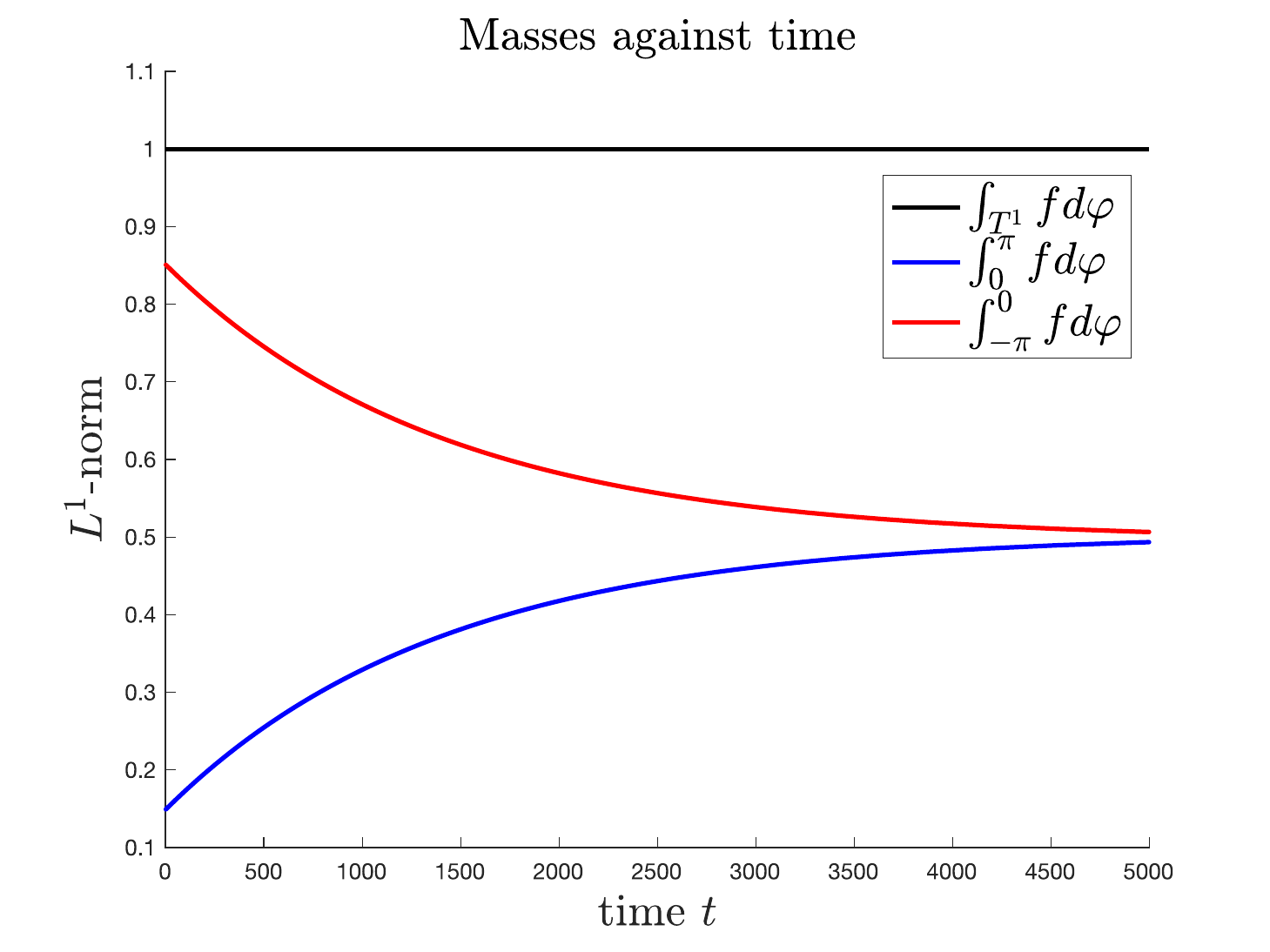}
		\end{subfigure}
		\caption{Initial conditions supported on $(-3\pi/4,-\pi/4)$ and $(\pi/4,3\pi/4)$, as well as in a very small interval contained in $(-\pi/4,\pi/4)$. $\Gamma$ has one connected component. \emph{Left:} Initial condition (solid dark blue), $f$ after 2500 time-steps (dashed red), $f$ after 5000 time-steps (solid black) and the equilibrium $f_\infty$ (dotted light blue). \emph{Right:} Total mass conservation (black), masses of the positive (dark blue) and negative (red) part of the torus, which are also conserved quantities. }
		\label{c4:I12}
	\end{figure}
	
	\paragraph{The graph $\Gamma$ has two connected components:}
	For the simulations corresponding to Figure \ref{c4:I14} initial data only positive in the intervals~$(-3\pi/4,-\pi/4)$ and $(\pi/4,3\pi/4)$ was chosen. This causes the graph $\Gamma$ to have two connected components $\Gamma_-$, supported in $(-\pi,0)$ and $\Gamma_+$, supported in $(0,\pi)$. The masses in the corresponding sets of vertices $\mathcal{V}_{\pm}$ were chosen different from each other. \\
	Again, the left part of Figure \ref{c4:I14} shows snapshots of the distribution function $f$ at different times together with the equilibrium 
	\[
	f_{\infty}(\vp) = 2\begin{cases} \mu(\vp) \int_{-\pi}^0 f(\tilde{\vp}) \, d \tilde{\vp}, \quad \vp \in (-\pi,0] \\  \mu(\vp) \int_0^{\pi} f(\tilde{\vp}) \, d \tilde{\vp}, \quad \vp \in (0,\pi].\end{cases}
	\]
	In the second row the total mass $\int_{\T} f \, d\vp =1$ as well as $\int_{-\pi}^0 f \, d\vp$ and $\int_0^{\pi} f \, d\vp$ are plotted against time, which shows mass conservation in $\mathcal{V}_{\pm}$. \\
	The simulation was carried out for $\Delta t=0.01$ and 250 time-steps.
	
	\begin{figure}[h!]
		\centering
		\begin{subfigure}{0.48\textwidth} 
			\includegraphics[width=\textwidth]{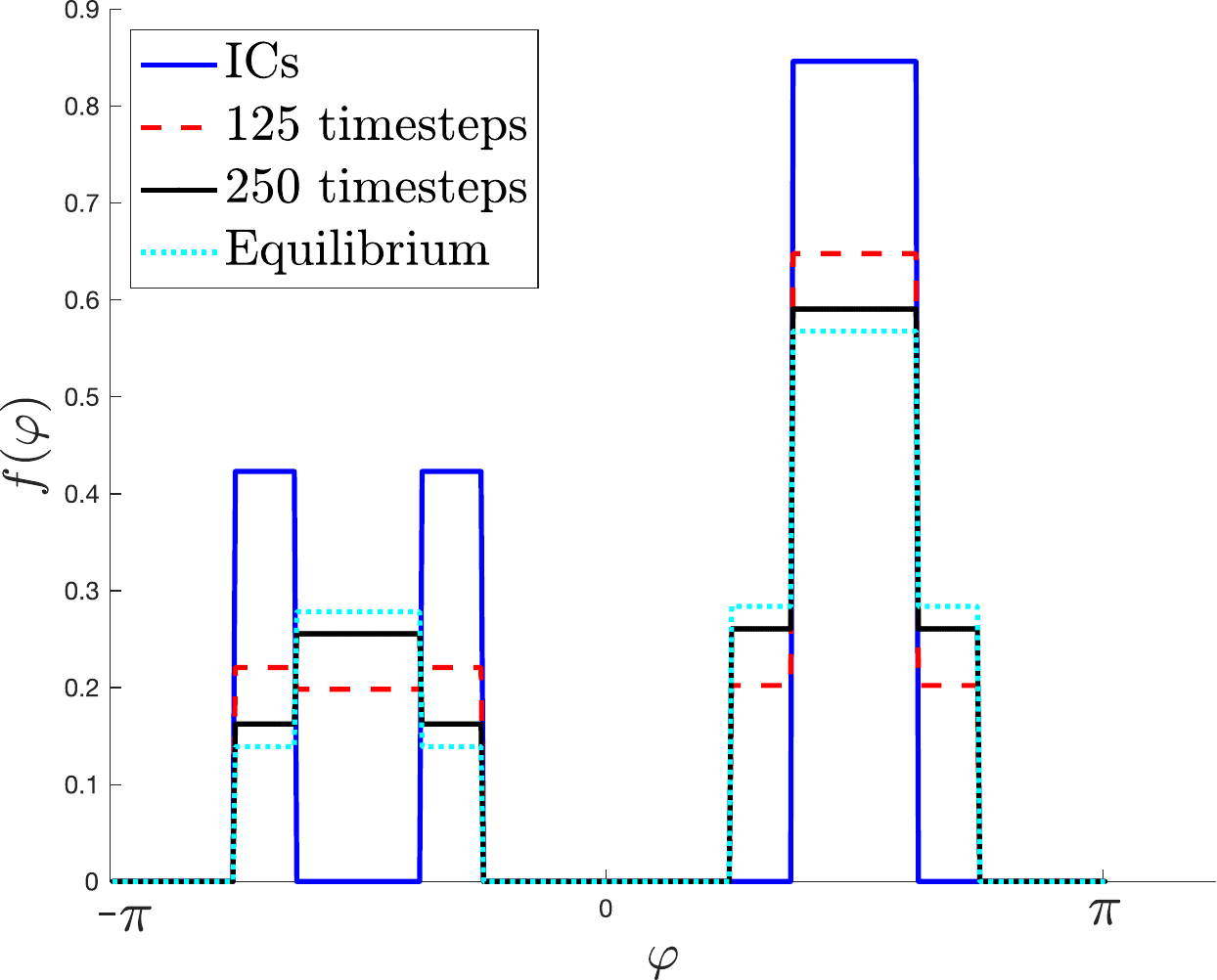}
		\end{subfigure}
		\hspace{1em}
		\begin{subfigure}{0.48\textwidth} 
			\includegraphics[width=\textwidth]{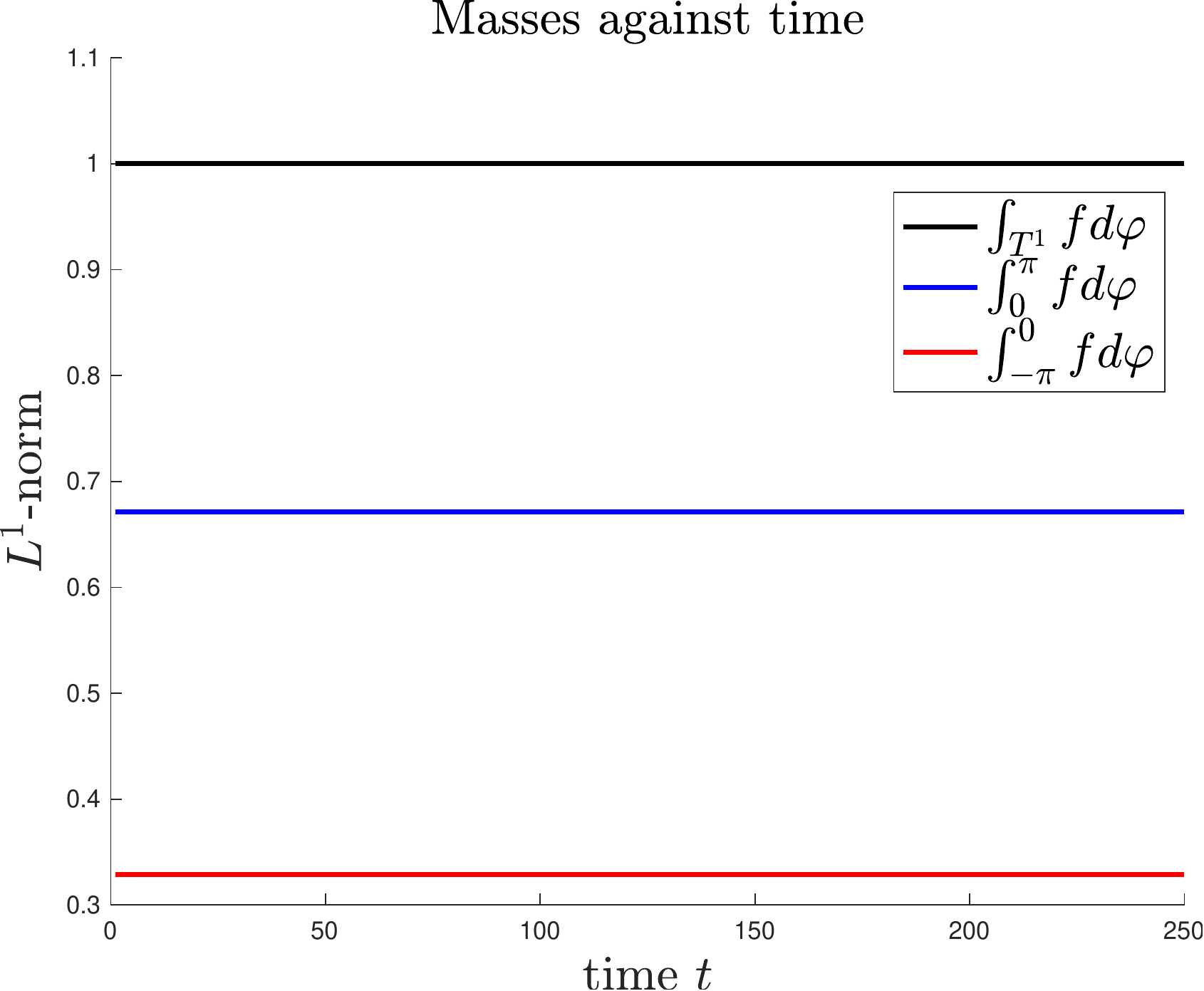}
		\end{subfigure}
		\caption{Initial conditions supported on $(-3\pi/4,-\pi/4)$ and $(\pi/4,3\pi/4)$, vacuum else. $\Gamma$ has two connected components $\Gamma_-$ with $\mathcal{V}_-\subset (-\pi,0)$ and $\Gamma_+$ with $\mathcal{V}_+\subset (0,\pi)$ . \emph{Left:} Initial condition (solid dark blue), $f$ after 125 time-steps (dashed red), $f$ after 250 time-steps (solid black) and the equilibrium (dotted light blue). \emph{Right:} Total mass conservation (black), masses of the positive (dark blue) and negative (red) part of the torus, which are different initially, but converge to the same value.  }
		\label{c4:I14}
	\end{figure}
	
	\appendix
	
	\section{Well-posedness in Wasserstein distance}\label{appendix-wasserstein}
	
	We recall that the Wasserstein-1 distance between two probability measures~$\mu$ and~$\nu$ on the compact metric space~$\ms$ is given, thanks to the Kantorovich-Rubinstein duality \cite{V}, by
	\begin{equation}
		W_1(\mu,\nu)=\sup\Big(\int_{\ms}\psi\,\md \mu-\int_{\ms}\psi\,\md \nu\Big),\label{defW1}
	\end{equation}
	where the supremum is taken over all 1-Lipschitz functions~$\psi$. And in our case of a compact space~$\ms$ the topology given by the Wasserstein-1 distance on~$\mathbb{P}(\ms)$ (the set of probability measures on~$\ms$) corresponds to the topology of weak convergence of measures. 
	
	\begin{proposition} We suppose that the collision kernel~$b$ is Lipschitz with a Lipschitz coefficient~$\lambda>0$. We denote by~$L$ a bound on the diameter of~$\ms$ and by~$M$ a bound on the collision kernel~$b$. Then, if~$f$ and~$\tf$ are two solutions to the reversal collision dynamics with respective initial conditions~$f_I$ and~$\tf_I$, we have the following global stability estimate with respect to the initial conditions:
		\begin{equation*}
			\forall t\geqslant0,\quad W_1(f(t,\cdot),\tf(t,\cdot))\leqslant e^{\lambda Lt}\,C(t)\,W_1(f_I,\tf_I),
		\end{equation*}
		where the coefficient~$C(t)$ is explicitly given by~$C(t)=1+(M+5\lambda L)t+2t^2\lambda ML$.
	\end{proposition}
	
	\begin{proof}
		Since~$f$ is a solution, thanks to Theorem~\ref{thm-existence-uniqueness} it is of the form $(1+h)\mu$ with~$\mu=\frac12(f_I+f_I^\downarrow)$ and~$h\in C([0,\infty),L^\infty(\mu))$. Therefore~$f$ belongs to~$C([0,\infty),\mathbb{P}(\ms))$. Using the fact that~$h$ is a fixed point of the mild formulation~\eqref{eq-fixed-point}, we get that $f$ is a fixed point of the map~$\Psi_{f_I}$, where~$\Psi_{f_I}(f)$ is given by the following formula, given for all~$\psi\in C(\ms)$ and~$t\in[0,\infty)$:
		\begin{equation*}
			\int_{\ms}\psi(x)\md \Psi_{f_I}(f)(t,x)=\int_{\ms}e^{-2tB_\mu(x)}\psi(x)\md f_I(x)+\int_0^t\int_{\ms}e^{-2(t-s)B_\mu(x)}B_{f(s,\cdot)}(x)\psi(x)\md \mu(x)\md s,
		\end{equation*}
		where, for any~$\nu\in\mathbb{P}(\ms)$, we write~$B_{\nu}(x)=\int_{\ms}b(x,x_*)\md \nu_*$. This gives a definition of~$\Psi_{f_I}(f)$ as an element of~$C([0,\infty),\mathbb{P}(\ms))$, by Riesz-Markov-Kakutani representation theorem.
		
		We start by proving the following estimate:
		\begin{equation}
			W_1(\Psi_{f_I}(f),\Psi_{\tf_I}(\tf))\leqslant\int_0^t\lambda LW_1(f(s,\cdot),\tf(s,\cdot))\md s +C(t) \,W_1(f_I,\tf_I).\label{estW1W1}
		\end{equation}
		
		We notice that in the definition~\eqref{defW1} of the Wasserstein distance, if we fix~$x_0\in\ms$, we can restrict the supremum over functions~$\psi$ which are~$1$-Lipschitz and such that~$\Psi(x_0)=0$. From now on we fix such a~$x_0$ and~$\psi$ and want to estimate, at a fixed time~$t>0$, the quantity
		\begin{equation}
			\int_{\ms}\psi(x)\md \Psi_{f_I}(f)(t,x)-\int_{\ms}\psi(x)\md \Psi_{\tf_I}(\tf)(t,x):=A_1+A_2+\int_0^t(A_3(s)+A_4(s)+A_5(s))\md s,\label{splitEstimate}
		\end{equation}
		where we have split it thanks to the five following expressions:
		\begin{align*}
			A_1&=\int_{\ms}e^{-2tB_\mu(x)}\psi(x)\md f_I(x)-\int_{\ms}e^{-2tB_\mu(x)}\psi(x)\md \tf_I(x),\\
			A_2&=\int_{\ms}\big(e^{-2tB_{\mu}(x)}-e^{-2tB_{\tilde{\mu}}(x)}\big)\psi(x)\md \tf_I(x),\\
			A_3(s)&=\int_{\ms}e^{-2(t-s)B_\mu(x)}B_{f(s,\cdot)}(x)\psi(x)\md \mu(x)-\int_{\ms}e^{-2(t-s)B_\mu(x)}B_{f(s,\cdot)}(x)\psi(x)\md \tilde{\mu}(x),\\
			A_4(s)&=\int_{\ms}e^{-2(t-s)B_\mu(x)}(B_{f(s,\cdot)}(x)-B_{\tf(s,\cdot)}(x))\psi(x)\md \tilde{\mu}(x)),\\
			A_5(s)&=\int_{\ms}\big(e^{-2(t-s)B_{\mu}(x)}-e^{-2(t-s)B_{\tilde{\mu}}(x)}\big)B_{\tf(s,\cdot)}(x)\psi(x)\md \tilde{\mu}(x).
		\end{align*}
		
		We first notice that since~$b$ is~$\lambda$-Lipschitz, then for all~$\nu\in\mathbb{P}(\ms)$,~$B_\nu$ is also~$\lambda$-Lipschitz. Then since~$b$ is bounded by~$M$,~$B_\nu$ is also bounded by~$M$. And finally we have that~$|\psi|$ is bounded by~$L$ since~$|\psi(x)|=|\psi(x)-\psi(x_0)|\leqslant d(x,x_0)$. Therefore we have for~$x,y\in\ms$
		\begin{align*}
			e^{-2tB_\mu(x)}\psi(x)-e^{-2tB_\mu(y)}\psi(y)&=(e^{-2tB_\mu(x)}-e^{-2tB_\mu(y)})\psi(x)+e^{-2tB_\mu(y)}(\psi(x)-\psi(y))\\
			&\leqslant2t\lambda d(x,y)L+d(x,y),
		\end{align*}
		Therefore the function~$x\mapsto e^{-2tB_\mu(x)}\psi(x)$ is~$(1+2t\lambda L)$-Lipschitz, and this provides the estimate
		\begin{equation*}
			A_1\leqslant(1+2t\lambda L)W_1(f_I,\tf_I).
		\end{equation*}
		Similarly, the function~$x\mapsto e^{-2(t-s)B_\mu(x)}B_{f(s,\cdot)}(x)\psi(x)$ is~$(M+\lambda L+2(t-s)\lambda ML)$-Lipschitz, and we obtain
		\begin{equation*}
			A_3(s)\leqslant(M+\lambda L+2(t-s)\lambda ML)W_1(\mu,\tilde{\mu}).
		\end{equation*}
		
		Furthermore, still thanks to the fact that~$b$ is~$\lambda$-Lipschitz, we have for all~$\nu,\tilde{\nu}\in\mathbb{P}(\ms)$:
		\begin{equation*}
			B_{\nu}(x)-B_{\tilde{\nu}}(x)\leqslant\lambda W_1(\nu,\tilde{\nu}).
		\end{equation*}
		Therefore this gives the estimates
		\begin{align*}
			A_2&\leqslant2t\lambda W_1(\mu,\tilde{\mu})L,\\
			A_4(s)&\leqslant\lambda LW_1(f(s,\cdot),\tf(s,\cdot),\\
			A_5(s)&\leqslant2(t-s)\lambda W_1(\mu,\tilde{\mu})ML.
		\end{align*}
		
		Since for any~$1$-Lipschitz function~$\psi$, the function~$\frac12(\psi+\psi^\downarrow)$ is also~$1$-Lipschitz, we get that 
		\begin{equation*}\int_{\ms}\psi\md \mu-\int_{\ms}\psi\md \tilde{\mu}=\int_{\ms}\frac12(\psi+\psi^\downarrow)\md f_I-\int_{\ms}\frac12(\psi+\psi^\downarrow)\md f_I.
		\end{equation*}
		Therefore we obtain~$W_1(\mu,\tilde{\mu})\leqslant W_1(f_I,\tf_I)$. Thanks to these estimates, we obtain that
		\begin{align*}
			A_1&+A_2+\int_0^t(A_3(s)+A_5(s))\md s\\
			&\leqslant(1+2t\lambda L+2t\lambda L+(M+\lambda L)t+t^2\lambda ML+t^2\lambda ML)W_1(f_I,\tf_I)=C(t)W_1(f_I,\tf_I).
		\end{align*}
		Therefore the expression given by~\eqref{splitEstimate} is bounded by the right-hand side of the inequality~\eqref{estW1W1}. Since this is true for all~$1$-Lipschitz function~$\psi$ such that~$\psi(x_0)=0$, this gives the inequality~\eqref{estW1W1}.
		
		Finally, since~$\Psi_{f_I}(f)=f$ and~$\Psi_{\tf_I}(\tf)=\tf$, the inequality~\eqref{estW1W1} becomes an integral Grönwall estimate, which gives the final result.
		
		Notice that the inequality~\eqref{estW1W1} could also be used to directly prove existence of a fixed point of~$\Psi_{f_I}$ (and thus a solution) in the space~$C([0,T],\mathbb{P}(\ms))$, restricted to time dependent probability measures~$f$ such their symmetric part is constant equal to~$\mu$, in the same manner as for the proof of Theorem~\ref{thm-existence-uniqueness}.
	\end{proof}

\end{document}